\documentclass[11pt]{amsart}
\usepackage{amssymb,amscd,  latexsym, graphicx, mathrsfs, enumerate}
\usepackage{color}

\newcommand{\nc}{\newcommand}
\numberwithin{equation}{section}
\newtheorem{theorem}{Theorem}[section]
\newtheorem{prop}[theorem]{Proposition}
\newtheorem{importnota}[theorem]{Important Notation}
\newtheorem{prblm}[theorem]{Problem}
\newtheorem{notation}[theorem]{Notation}
\newtheorem{caution}[theorem]{Caution}
\newtheorem{remark}[theorem]{Remark}
\newtheorem{lemma}[theorem]{Lemma}
\newtheorem{construction}[theorem]{Construction}
\newtheorem{corollary}[theorem]{Corollary}
\newtheorem{example}[theorem]{Example}
\newtheorem{conclusion}[theorem]{Conclusion}
\newtheorem{triviality}[theorem]{Triviality}
\newtheorem{proto}[theorem]{Prototype Quasifibration}
\newtheorem{cauex}[theorem]{Cautionary Example}
\newtheorem{propositiondef}[theorem]{Proposition-Definition}
\newtheorem{subth}{Nuisance}[theorem]
\newtheorem{ssubth}{ }[subth]
\newtheorem{conjecture}[theorem]{Conjecture}
\newtheorem{sidest}[theorem]{Side Story}
\newtheorem{miniexample}[theorem]{Example}
\theoremstyle{definition}
\newtheorem{defin}[theorem]{Definition}

\nc\tri[1]{\begin{triviality}}
\nc\side[1]{\begin{sidest}}
\nc\conj[1]{\begin{conjecture}}
\nc\prodef[1]{\begin{propositiondef}}
\nc\prt[1]{\begin{proto}}
\nc\lem[1]{\begin{lemma}}
\nc\sblm[1]{\begin{sublemma}}
\nc\pro[1]{\begin{prop}}
\nc\thm[1]{\begin{theorem}}
\nc\cor[1]{\begin{corollary}}
\nc\dfn[1]{\begin{defin}}
\nc\sthm[1]{\begin{subth}}
\nc\exm[1]{\begin{example}}
\nc\miniexm[1]{\begin{miniexample}}
\nc\plm[1]{\begin{prblm}}
\nc\rmk[1]{\begin{remark}}
\nc\subrmk[1]{\begin{subremark}}
\nc\ntn[1]{\begin{notation}}
\nc\cau[1]{\begin{caution}}
\nc\imn[1]{\begin{importnota}}
\nc\cax[1]{\begin{cauex}}
\nc\con[1]{\begin{construction}}
\nc\ssthm[1]{\begin{ssubth}}
\nc\cnc[1]{\begin{conclusion}}
\nc\elem{\end{lemma}}
\nc\esblm{\end{sublemma}}
\nc\eside{\end{sidest}}
\nc\econj{\end{conjecture}}
\nc\eprodef{\end{propositiondef}}
\nc\eprt{\end{proto}}
\nc\ethm{\end{theorem}}
\nc\ecor{\end{corollary}}
\nc\edfn{\end{defin}}
\nc\esthm{\end{subth}}
\nc\epro{\end{prop}}
\nc\etri{\end{triviality}}
\nc\eexm{\end{example}}
\nc\eminiexm{\end{miniexample}}
\nc\ermk{\end{remark}}
\nc\subermk{\end{subremark}}
\nc\eplm{\end{prblm}}
\nc\ecau{\end{caution}}
\nc\ecax{\end{cauex}}
\nc\eimn{\end{importnota}}
\nc\entn{\end{notation}}
\nc\econ{\end{construction}}
\nc\ecnc{\end{conclusion}}
\nc\essthm{\end{ssubth}}

\newcommand{\C}{\mathbb{C}}
\newcommand{\R}{\mathbb{R}}
\newcommand{\Q}{\mathbb{Q}}

\newcommand{\Z}{\mathbb{Z}}

\newcommand{\N}{\mathbb{N}}

\newcommand{\F}{\mathbb{F}}
\newcommand{\A}{\mathbb{A}}
\newcommand{\E}{\mathcal{E}}

\newcommand{\e}{\varepsilon}

\newcommand{\diag}{{\rm diag}}
\newcommand{\G}{\Gamma}

\newcommand{\ds}{\displaystyle}

\newcommand{\bF}{\overline{\F}}

\newcommand{\lra}{\longrightarrow}

\newcommand{\sgn}{{\rm sgn}}

\newcommand{\ve}{\varepsilon}

\renewcommand{\Bbb}{\mathbb}

\title[Artin representations for $GSp_4$]
{A Conditional construction of Artin representations for\\ 
real analytic Siegel cusp forms of weight $(2,1)$}
\author{Henry H. Kim and Takuya Yamauchi}
\dedicatory{To Jim Cogdell on his 60th birthday}
\thanks{The first author is partially supported by NSERC. The second author
is partially supported by JSPS Grant-in-Aid for Scientific Research No.23740027 and JSPS Postdoctoral 
Fellowships for Research Abroad No.378.}

\address{Henry H. Kim \\
Department of mathematics \\
 University of Toronto \\
Toronto, Ontario M5S 2E4, CANADA \\
and Korea Institute for Advanced Study, Seoul, KOREA}
\email{henrykim@math.toronto.edu}

\address{Takuya Yamauchi \\
Department of mathematics, Faculty of Education\\
Kagoshima University\\
Korimoto 1-20-6 Kagoshima 890-0065, JAPAN and 
Department of mathematics \\
 University of Toronto \\
Toronto, Ontario M5S 2E4, CANADA}
\email{yamauchi@edu.kagoshima-u.ac.jp or tyama@math.toronto.edu}

\begin{document}

\begin{abstract}
Let $F$ be a vector-valued real analytic Siegel cusp eigenform of weight $(2,1)$ with the eigenvalues $-\frac 5{12}$ and $0$ for the two generators of the center of the algebra consisting of all $Sp_4(\R)$-invariant differential operators on the Siegel upper half plane of degree 2. Under natural assumptions in analogy of holomorphic Siegel cusp forms, we construct a unique symplectically odd Artin representation
$\rho_F: G_\Q\lra GSp_4(\C)$ associated to $F$. For this, we develop the arithmetic theory of vector-valued real analytic Siegel modular forms.
Several examples which satisfy these assumptions are given by using various transfers and automorphic induction.
\end{abstract}
\maketitle
\tableofcontents

\section{Introduction}\label{introduction}

Let $G$ be a quasi-split reductive group over $\Q$ and ${}^L G$ be the $L$-group.
The strong Artin conjecture asserts that an irreducible continuous complex representation $\rho: G_{\Bbb Q}={\rm Gal}(\overline{\Bbb Q}/\Bbb Q)\longrightarrow 
{}^L G$ corresponds to an automorphic representation of $G$. We will call such an irreducible complex representation an Artin representation.
Conversely, it is expected that there exists an Artin representation associated to a given 
automorphic representation with certain special properties (such as a special infinity type).

In \cite{d&s}, Deligne and Serre associated an odd irreducible Artin representation $\rho: G_\Q\lra GL_2(\C)$ to any elliptic cusp form of weight one which is a Hecke eigenform. 
Conversely, the strong Artin conjecture is now proved for any odd irreducible 2-dimensional Artin representation: Solvable cases have been known for a long time 
\cite{langlands}, \cite{tunnell}. The icosahedral case is proved in \cite{k&w1} and \cite{kisin}.
Note that the automorphy of an even icosahedral Artin representation $\rho: G_\Q\lra GL_2(\C)$ is still open. 

Contrary to the case $GL_2/\Q$, we can show that there are no holomorphic Siegel cusp forms of weight $(2,1)$
which give rise to Artin representations (Proposition \ref{no-hol}). Therefore we have to look for the corresponding objects inside real analytic Siegel modular forms. 
This fact makes the situation much more difficult as in the case of Maass cusp forms for $GL_2/\Q$. 

In this paper under natural assumptions, we will construct an Artin representation
$\rho: G_\Q\lra GSp_4(\C)$
associated to a real analytic Siegel cusp form of weight $(2,1)$. More precisely, let $F$ be such a form of weight $(2,1)$, level $N$ with central character $\varepsilon$. (See Section \ref{class} for the notion of level and a central character of real analytic Siegel modular forms.) Assume that $F$ has the eigenvalues $-\frac 5{12}$ and $0$ for the elements $\Delta_1$ and $\Delta_2$ of degree 2 and 4 respectively which generate the center of the algebra consisting of all $Sp_4(\R)$-invariant differential operators on the Siegel upper half plane. (See Section \ref{adelic} and the equation (\ref{generators}) for the precise definition of $\Delta_1$ and $\Delta_2$.)
We denote by $S_{(2,1)}(\G(N),-\frac 5{12},0)$, the space consisting of all such $F$. 
Let $t_1={\rm diag}(1,1,p,p)$ and $t_2={\rm diag}(1,p,p^2,p)$ for a prime $p$ and denote by $T_{i,p}$, the Hecke operator 
corresponding to $t_i$ for $i=1,2$ (see Section \ref{integrality}).  
Assume that $F$ is an eigenform for $T_{1,p}$ and $T_{2,p}$, $p\nmid N$ with eigenvalues $a_{1,p}$ and $a_{2,p}$ respectively. 
We define the Hecke polynomial at $p$ by 
$$H_p(T):=1-a_{1,p}T+\{pa_{2,p}+(1+p^{-2})\e(p^{-1})\}T^2-a_{1,p}\e(p^{-1})T^3+\e(p^{-1})^2T^4.$$
Let $\pi_F=\pi_{\infty}\otimes \otimes_p' \pi_p$ be the cuspidal automorphic representation of $GSp_4(\A)$ attached to $F$. 
Then by Theorem \ref{infinity}, $\pi_\infty$ is the full induced representation ${\rm Ind}_B^G\, \chi(1,\sgn,\sgn)$, which is irreducible. In particular, it is tempered and generic. It is a totally degenerate limit of discrete series in the sense of \cite{CK}. It is also denoted as $D_{(1,0)}[0]$ in \cite{mor}. We emphasize here that $\pi_{\infty}$ is not a limit of holomorphic discrete series and it has the minimal $K_\infty$-type $(1,0)$. 

We denote by $\Q_F=\Q(a_{1,p},a_{2,p},\e(p^{-1})\,:\, p\nmid N)$, the Hecke field of $F$. 
We also define the Hecke algebra $\mathbb{T}_\Q$ over $\Q$ as in Section \ref{rat}. 
We then make the following assumptions: 
\begin{enumerate}
\item[(TR)] the existence of the transfer of automorphic representations of $GSp_4$ to $GL_4$ (Section \ref{arthur});
\item[(Gal)] the existence of mod $\ell$ Galois representation attached to $F$ (Conjecture \ref{galois}); 
\item[(Rat)] rationality of the subspace $\langle TF\ |\ T\in \mathbb{T}_\Q  \rangle_\C$ of $S_{(2,1)}(\G(N),-\frac 5{12},0)$ (Section \ref{rationality}); 
\item[(Int)] the integrality of Hecke polynomials $H_p(T)$ of $F$ for all but finitely many prime $p\nmid N$ (Section \ref{ran-sel}).
\end{enumerate}
The assumptions (Gal), (Rat), and (Int) are valid for holomorphic Siegel cusp forms of weight $(k_1,k_2)$, $k_1\geq k_2\geq 2$ (\cite{taylor-thesis}, \cite{taylor}).
We prove 

\begin{theorem} (Main Theorem)\label{artin-rep}
Let $F$ be as above, and assume (TR), (Gal), (Rat), and (Int).
Then there exists the Artin representation 
$\rho_F: G_\Q\lra GSp_4(\C)$ which is unramified outside primes dividing $N$ and symplectically odd, i.e. $\rho_{F}(c)\stackrel{\tiny{GSp_4(\C)}}{\sim }{\rm diag}(1,-1,-1,1)$ for 
the complex conjugation $c$ such that $\det(I_4-\rho_F({\rm Frob}_p)T)=H_p(T)$ 
for all $p\nmid N$. This representation is irreducible if and only if $\pi_F$ is not an endoscopic representation.
\end{theorem}
As a corollary to our main theorem, we see that $\pi_p$ is tempered for all $p$ (Corollary 9.2). 
\begin{remark}\label{weightchange}
Let $F$ be as in Theorem \ref{artin-rep}. Then $F$ cannot be a CAP form (Lemma \ref{CAP}).
We can show that there exists a real analytic Siegel modular form $F'$ of weight $(1,0)$ 
such that $\pi_{F'}$ and $\pi_F$ are nearly equivalent. A difference appears in the L-functions, namely, $L^N(s,\pi_F)=L^N(s,F)=L^N(s-1,F')=L^N(s,\pi_{F'})$. In fact, one can get $F$ from $F'$ by first multiplying $det(Y)$, and then by taking a differential operator, and vice versa. See Remark \ref{(1,0)} and Section \ref{ran-sel} for the details. However, only the form of weight $(2,1)$ gives rise to the Artin representation (Remark \ref{(2,1)}).
\end{remark}
For the proof of the main theorem, we follow the method of \cite{d&s}, in which the main ingredients are: 
\begin{enumerate}
\item the integrality of Hecke eigenvalues (Hecke polynomials) for all but finitely many primes (This is related to the fact that the infinity type is a limit of holomorphic discrete series. This uses algebraic geometric structures.);
\item the Rankin-Selberg L-function $L(s,\pi_f\times\tilde\pi_f)$, where $\pi_f$ is the cuspidal representation attached to a cusp form $f$ of weight one,
\item the classification of all semisimple subgroup of $GL_2(\F_\ell)$ for any odd prime $\ell$,  
\item the existence of mod $\ell$ Galois representations attached to any elliptic modular form of weight one. Here the key idea is to use Eisenstein series congruent to 1 mod $\ell$ to shift up the weight.
\end{enumerate}
The essentially same technique can be applied to the case of Hilbert modular forms.
(See \cite{rog&tunnell}, \cite{ohta} for any Hilbert cusp form with parallel weight one. For partial results on the automorphy of Galois representations, see
\cite{sasaki}, \cite{kas}, \cite{gee-kas}.)

%\smallskip
Contrary to the case $GL_2/\Q$, we cannot use the algebro-geometric techniques for our form $F$ since $F$ is non-holomorphic. Therefore we do not have ingredients (1) and (4) in the current situation. For these reasons we make assumptions (Gal), (Rat), and (Int). We will give various examples which satisfy these assumptions. 

For (3) in $GSp_4$ case, we study the classification of certain semisimple subgroups of $GSp_4(\F_\ell)$ for any odd prime $\ell$ in Section \ref{subgroups1}. This part could be 
simplified in terms of the theory of finite groups for classical groups. But there are no suitable references in the literature. In the upcoming paper \cite{KY}, we simplify the proof and generalize it to arbitrary semisimple subgroups of $GL_n(\F_\ell)$.

For (2), we apply the result of Arthur \cite{arthur} on the transfer from 
$GSp_4$ to $GL_4$ to obtain the automorphic representation $\Pi$ of $GL_4(\A)$ so that $L(s,\pi_F)=L(s,\Pi)$. 
The result of Arthur depends on the stabilization of the twisted trace formula for $GSp_4$, which is not done at this moment. 
We emphasize that we only need the transfer from $GSp_4$ to $GL_4$. In the upcoming paper \cite{KY}, we remove this assumption (TR).
Let $\tau$ be an automorphic representation of $GL_6(\A)$ such that $\tau_p\simeq\wedge^2 \Pi_p$ for all $p\ne 2,3$ \cite{Kim}. 
Then $\tau=(\Pi_5\otimes\varepsilon)\boxplus \varepsilon$, where $\varepsilon$ is the central character of $\pi_F$, and $\Pi_5$ is the automorphic representation of $GL_5(\A)$ which is a weak transfer of
$\pi_F$ to $GL_5$ corresponding to the $L$-group homomorphism $GSp_4(\Bbb C)\longrightarrow GL_5(\Bbb C)$, given by the second fundamental weight \cite{Kim1}.
Then we can use the Rankin-Selberg $L$-functions $L(s,\Pi\times\tilde\Pi)$ and $L(s,\Pi_5\times\tilde\Pi_5)$ to prove finiteness of a certain set. This is done in Section 6.

In Section \ref{sym-cube}, we obtain the existence of the real analytic Siegel cusp form of weight $(2,1)$ attached to 
the symmetric cube lift of elliptic eigen cuspform of weight 1. More precisely, let $f$ be an elliptic cusp form of weight 1 which is a Hecke eigenform. Let $\pi_f$ be the cuspidal representation of $GL_2/\Bbb Q$ attached to $f$. Then ${\rm Sym}^3(\pi_f)$ is an automorphic representation of $GL_4/\Bbb Q$ \cite{Kim-Sh}. Then by the result of Jacquet, Piatetski-Shapiro, and Shalika (cf. \cite{AS}, \cite{AS1}), there exists a generic cuspidal representation $\Pi$ of $GSp_4(\A)$, whose transfer to $GL_4(\A)$ is ${\rm Sym}^3(\pi_f)$. We can show that $\Pi_{\infty}$ is equivalent to ${\rm Ind}_B^G\, \chi(1,\sgn,\sgn)$. 
Hence we can find a real analytic Siegel cusp form of weight $(2,1)$ with the eigenvalues $-\frac 5{12}$ and $0$ for the generators $\Delta_1$ and $\Delta_2$ such that $\pi_F\simeq \Pi$. This provides infinitely many examples of Siegel cusp forms of weight $(2,1)$ with integral Hecke polynomials. Note that this is an unconditional result. 
If $\rho_f: G_{\Bbb Q}\longrightarrow GL_2(\Bbb C)$ is the Artin representation associated to $f$ by Deligne-Serre theorem \cite{d&s}, then ${\rm Sym}^3(\rho_f)$ is the Artin representation associated to $F$.
Finally we state the strong Artin conjecture related to the automorphy of Artin representations for $GSp_4$: 
\begin{conjecture}
Let $\rho: G_\Q\lra GSp_4(\C)$ be a symplectically odd Artin representation whose image 
does not factor through, up to conjugacy in $GSp_4(\C)$, the Levi factor of 
any parabolic subgroup of $GSp_4(\C)$. 
Then there exists a real analytic Siegel cusp modular form $F$ of weight $(2,1)$ with the eigenvalues $-\frac 5{12}$ and $0$ for the generators $\Delta_1$ and $\Delta_2$ 
so that $\rho_F\sim \rho$. 
\end{conjecture}
As in $GL_2$ case, we consider $\bar\rho: G_\Q\lra PGSp_4(\C)$, the composition of $\rho$ and the canonical projection. 
Then $PGSp_4(\C)\simeq SO_5(\C)$ and the finite subgroups of $SO_5(\C)$ have been classified (cf. \cite{martin}).
One expects a case by case analysis for different finite subgroups.
In fact, K. Martin \cite{martin} showed the strong Artin conjecture when ${\rm Im}(\bar\rho)$ is a solvable group, $E_{16}\rtimes C_5$, where $E_{16}$ is the
elementary abelian group of order 16 and $C_5$ is the cyclic group of order 5. In Section \ref{solvable}, we show 
that K. Martin's explicit examples
give rise to symplectically odd Artin representations. Hence we obtain examples of real analytic Siegel cusp forms of weight $(2,1)$ 
with integral Hecke polynomials, which do not come from $GL_2$ forms.
\smallskip

\textbf{Acknowledgments.} We would like to thank J. Arthur, K. Gunji, S. Kudla, S. Kuroki, C-P. Mok, T. Moriyama,  T. Okazaki, R. Schmidt, and D. Vogan for helpful discussions. In particular, R. Schmidt \cite{schmidt1} helped us to realize that holomorphic Siegel cusp forms never give rise to Artin representations. We thank the referee for many helpful comments and corrections.

\section{Preliminaries on $GSp_4$}
Let $J=\begin{pmatrix} 0&I_2\\-I_2&0\end{pmatrix}$, and we realize the algebraic group $G:=GSp_4$ over $\Q$ as 
the subgroup of $GL_4$ consisting of all $g$ such that 
${}^tgJg=\nu(g) J$ for some $\nu(g)\in GL_1$. 
Let $\nu:G=GSp_4\lra GL_1$ be the similitude character defined by sending $g$ to $\nu(g)$. 
Let $Sp_4:={\rm Ker}(\nu)$. Let $B$ be the Borel subgroup of $GSp_4$, $T$ the maximal torus, 
and $U$ the unipotent radical of $B$:
\begin{eqnarray*}
T &=& \{ t=t(t_1,t_2,t_0)=\diag(t_1,t_2, t_0 t_1^{-1}, t_0 t_2^{-1}) | \, t_0, t_1, t_2\in GL_1\},\\
U &=& \left\{ \begin{pmatrix} I_2& A\\ 0&I_2\end{pmatrix} \begin{pmatrix} C&0\\0& C'\end{pmatrix} \Bigg| \, A=\begin{pmatrix} a&b\\b&c\end{pmatrix}, 
C=\begin{pmatrix} 1&d\\0&1\end{pmatrix}, C'=\begin{pmatrix} 1&0\\-d&1\end{pmatrix}\right\}
\end{eqnarray*}
The simple roots are $\alpha(t(t_1,t_2,t_0))=t_1t_2^{-1}$ and $\beta(t(t_1,t_2,t_0))=t_2^2t_0^{-1}$. The coroots are $\alpha^{\vee}(x)=t(x,x^{-1},1)$ and $\beta^{\vee}(x)=t(1,x,1)$. Let 
$s_1=\begin{pmatrix}
0 & 1 & 0 & 0\\
1 & 0 & 0 & 0\\
0 & 0 & 0 & 1\\
0 & 0 & 1 & 0
\end{pmatrix},\quad
s_2=\begin{pmatrix}
1 & 0 & 0 & 0\\
0 & 0 & 0 & 1\\
0 & 0 & 1 & 0\\
0 & -1 & 0 & 0
\end{pmatrix}
$ 
be the representatives of generators of the Weyl group $N_G(T)/T$ which 
correspond to $\alpha$ and $\beta$ respectively. 

Note that $Z_G=\{ aI_{4}: a\in GL_1 \}$ and $\nu(aI_{4})=a^2$. For any place $p\le \infty$ of $\Q$, 
a character of $\chi$ of $T(\Q_p)$ is given by
$\chi=\chi(\chi_1,\chi_2,\sigma)$ where $\chi_i,\sigma$ are characters of $\Bbb Q^{\times}_p$ so that $\chi(t(t_1,t_2,t_0))=\chi_1(t_1)\chi_2(t_2)\sigma(t_0)$. Note also that the dual group of $G$ is $\widehat{G}=GSp_4(\Bbb C)$. 

Note the Weyl group action; $s_1 : t(t_1,t_2,t_0)\longmapsto t(t_2,t_1,t_0)$ and $s_2: t(t_1,t_2,t_0)\longmapsto t(t_1,t_2^{-1}t_0,t_0)$. 

Let $P=M_PN_P$ (resp. $Q=M_QN_Q$) be the Siegel (resp. Klingen) parabolic subgroup of $GSp_4$ containing $B$, where
$$
M_P=\left\{\begin{pmatrix} A&0\\0& u {}^t A^{-1}\end{pmatrix}\, : A\in GL_2,\, u\in GL_1\right\}\simeq GL_2\times GL_1,\,
N_P=\left\{ \begin{pmatrix} I_2&S\\0& I_2\end{pmatrix} : \, S=\begin{pmatrix} a&b\\b&c\end{pmatrix} \right\},
$$
$$N_Q=\left\{ \begin{pmatrix} I_2& A\\ 0&I_2\end{pmatrix} \begin{pmatrix} C&0\\0& C'\end{pmatrix} \Bigg| \, A=\begin{pmatrix} a&b\\b&0\end{pmatrix}, 
C=\begin{pmatrix} 1&d\\0&1\end{pmatrix}, C'=\begin{pmatrix} 1&0\\-d&1\end{pmatrix}\right\},
$$
\begin{eqnarray*}
M_Q=\left\{\begin{pmatrix} a'&{}&{}&{}\\{}&1&{}&{}\\{}&{}&u{a'}^{-1}&{}\\{}&{}&{}&u\end{pmatrix} \begin{pmatrix} 1&{}&{}&{}\\{}&a&{}&b\\{}&{}&1&{}\\{}&c&{}&d\end{pmatrix},\, \begin{pmatrix} a&b\\c&d\end{pmatrix}\in SL_2 \right\}.
\end{eqnarray*}

\section{Vector-valued real analytic Siegel modular forms}
In this section, we shall discuss vector-valued real analytic Siegel modular forms in various settings. 
Since there are no references in dealing with vector-valued real analytic Siegel modular forms, we will develop the definition and their basic properties by imitating Section 1 to 3 of \cite{hori}. 
We refer to \cite{borel&jacquet} for the adelic setting. 

\subsection{Classical real analytic Siegel modular forms}\label{class}
Let $\mathcal{H}_2=\{Z\in M_2(\C)|\ {}^tZ=Z,\  {\rm Im}(Z)>0\}$ be the Siegel 
upper half-plane. 
For a pair of non-negative integers $\underline{k}=(k_1,k_2)$, $k_1\ge k_2$ (note that $k_2$ might be negative in real analytic case), we define the 
algebraic representation $\lambda_{\underline{k}}$ of $GL_2$ by 
$$V_{\underline{k}}={\rm Sym}^{k_1-k_2}{\rm St}_2\otimes {\rm det}^{k_2} {\rm St}_2,
$$ 
where ${\rm St}_2$ is the standard representation of dimension 2 with the basis $\{e_1,e_2\}$. More explicitly, if $R$ is any ring, 
then $V_{\underline{k}}(R)=\ds\bigoplus_{i=0}^{k_1-k_2}Re^{k_1-k_2-i}_1\cdot e^i_2$ and for 
$g=\begin{pmatrix}
a & b \\
c & d
\end{pmatrix}
\in GL_2(R)$,  $\lambda_{\underline{k}}(g)$ acts on $V_{\underline{k}}(R)$ by 
$$g\cdot e^{k_1-k_2-i}_1\cdot e^i_2:=\det(g)^{k_2}(ae_1+be_2)^{k_1-k_2-i}\cdot (ce_1+de_2)^i.$$
We identify $V_{\underline{k}}(R)$ with $R^{\oplus (k_1-k_2+1)}$,
and
$\lambda_{\underline{k}}(g)$ with the representation matrix of $\lambda_{\underline{k}}(g)$ 
with respect to the above basis.
We have the action and the automorphy factor $J$ by
\begin{equation}\label{sp4-action}
\gamma Z=(AZ+B)(CZ+D)^{-1}, \quad J(\gamma,Z)=CZ+D\in GL_2(\C),
\end{equation}
for
$\gamma=\begin{pmatrix}
A& B\\
C& D
\end{pmatrix}
\in Sp_4(\R)$ and $Z\in \mathcal{H}_2$.

For an integer $N\ge 1$, we define a principal congruence subgroup $\Gamma(N)$ to be 
the group consisting of the elements $g\in Sp_4(\Z)$ such that $g\equiv 1 \ {\rm mod}\ N$. 
For a parabolic subgroup $R\in \{B,P,Q\}$, let $\G_R(N)$ be the group consisting of the elements $g\in Sp_4(\Z)$ such that 
$g\ ({\rm mod}\ N) \in 
R(\Z/N\Z)$. 
%We denote by $\G^{{\rm para}}(N)$ the paramodular group defined in \cite{schmidt} (see also \cite{r&s}). 
For a $V_{\underline{k}}(\C)$-valued function $f$ on $\mathcal{H}_2$, the action of $\gamma \in G(\R)^+$ is defined by 
\begin{equation}\label{transformation}
F(Z)|[\gamma]_{\underline{k}}:=\lambda_{\underline{k}}(\nu(\gamma)J(\gamma,z)^{-1})F(\gamma Z).
\end{equation} 
The algebra of all $Sp_4(\R)$-invariant differential operators on $\mathcal{H}_2$ is 
isomorphic to $\C[\Delta_1,\Delta_2]$, the commutative polynomial ring of two variables (\cite{hori}), where
$\Delta_1$ is the degree 2 Casimir element, and $\Delta_2$ is the degree 4 element. (see Section 5 for the details.)

For an arithmetic subgroup $\Gamma$ of $Sp_4(\Q)$ and a finite order character 
$\chi:\Gamma\lra \C^\times$,  we say that a function $F:\mathcal{H}_2\lra V_{\underline{k}}(\C)$ is 
a real analytic Siegel modular form of weight $(k_1,k_2)$ with the character $\chi$ with respect to $\G$ if it satisfies

(i) $F$ is a $C^\infty$-function, 

(ii)  $F|[\gamma]_{\underline{k}}=\chi(\gamma)f$ for any $\gamma\in\Gamma$,  

(iii) $F$ is a common eigenform for $\Delta_1$ and $\Delta_2$, namely, $\Delta_i F=c_i F$ for some constants $c_i,i=1,2$, 

(iv) $F$ satisfies the growth condition, i.e.,
there exist a positive real number $C$ and $n\in \mathbb{N}$ such that 
for any linear functional $l:V_{\underline{k}}(\C)\lra \C$, 
$$|l(F(Z))|\le C(\sup\{{\rm tr}({\rm Im}Z),{\rm tr}({\rm Im}Z)^{-1} \})^n.$$
We denote by $M_{\underline{k}}(\Gamma,\chi,c_1,c_2)$ the space of such forms. 
By Harish-Chandra (see Theorem 1.7 of \cite{borel&jacquet}.) this space is finite dimensional.

For each parabolic subgroup $R$ of $Sp_4$ defined over $\Q$, we denote by $N_R$ the unipotent radical of $R$. 
Then we say $F\in M_{\underline{k}}(\Gamma,\chi,c_1,c_2)$ is a cusp form if 
$$\int_{(N_R(\Q)\cap \G)\backslash N_R(\R)} F(nZ)dn=0,\: \mbox{ for any parabolic subgroup $R$ defined over $\Q$}.$$ 
We denote by  $S_{\underline{k}}(\Gamma,\chi,c_1,c_2)$  the space of such cusp forms inside $M_{\underline{k}}(\Gamma,\chi,c_1,c_2)$. 

Similar to the holomorphic case (cf. \cite{evdokimov}), 
we shall define the Hecke operators on $M_{\underline{k}}(\Gamma(N),c_1,c_2)$:     
For any positive integer $n$ coprime to $N$, let 
$$\Delta_n(N):=\Bigg\{g\in M_4(\Z)\cap GSp_4(\Q)\ \Bigg|\ 
g\equiv \begin{pmatrix}
I_2 & 0   \\
0 & \nu(g)I_2 
\end{pmatrix}\, {\rm mod}\ N ,\ \nu(g)^{\pm1}\in\Z[\frac{1}{n}] \Bigg\}.
$$
For $m\in \Delta_n(N)$, we introduce the action of the Hecke operators on $M_{\underline{k}}(\Gamma(N),c_1,c_2)$ 
 by 
\begin{equation}\label{hecke-ope}
T_m F(Z):=\nu(m)^{\frac{k_1+k_2}{2}-3}\ds\sum_{\alpha\in \Gamma(N)\backslash\Gamma(N) m
\Gamma(N)} F(Z)|[(\nu(m)^{-\frac{1}{2}}\alpha]_{\underline{k}},
\end{equation}
and for any positive integer $n$, put  
$$T(n):=\sum_{m\in \G(N)\backslash \Delta_n(N)}T_m.
$$

We also consider the same actions on $S_{\underline{k}}(\G(N),c_1,c_2)$. 
For $t_1={\rm diag}(1,1,p,p),\ t_2={\rm diag}(1,p,p,p^2)$ for a prime $p$, put $T_{i,p}:=T_{t_i}\ i=1,2$ and 
fix $\widetilde{S}_{p,1}, \widetilde{S}_{p,p}\in Sp_4(\Z)$ so that 
$\widetilde{S}_{p,1}\equiv {\rm diag}(p^{-1},1,1,p)\ {\rm mod}\ N$ and $\widetilde{S}_{p,p}\equiv {\rm diag}(p^{-1},p^{-1},p,p)\ {\rm mod}\ N$. 
Then we see that 
\begin{equation}
T(p)=T_{1,p},\quad T^2_{1,p}-T(p^2)-p^2 \widetilde{S}_{p,p}=p\{T_{2,p}+(1+p^2)\widetilde{S}_{p,p}\}.
\end{equation}

The group $\G(N)$ contains the subgroup consisting of 
$\begin{pmatrix}
I_2& NS \\
0 & I_2
\end{pmatrix}$, 
$S={}^tS\in M_2(\Z)$. Hence for a given $F\in M_{\underline{k}}(\G(N),c_1,c_2)$, we have the Fourier expansion 
\begin{equation}\label{fourier}
F(Z)=\ds\sum_{T\in P(\Z)_{\ge 0}}A_F(T,Y)e^{\frac{2\pi\sqrt{-1}}{N}{\rm tr}(TX)}, \quad Z=X+Y \sqrt{-1}\in \mathcal H_2,
\end{equation}
 where $P(\Z)_{\ge 0}$ is the subset of $M_2(\Q)$ consisting of all symmetric matrices 
$\begin{pmatrix}
a& \frac{b}{2} \\
\frac{b}{2} & c
\end{pmatrix}$,
$a,b,c\in \Z$, which are semi-positive definite. 

\subsection{Hecke operators}\label{integrality}

The finite group $Sp_4(\Z/N\Z)\simeq Sp_4(\Z)/\G(N)$ acts on  
$M_{\underline{k}}(\G(N),c_1,c_2)$ by $F\mapsto F|[\tilde{\gamma}]_{\underline{k}}$ if 
we fix a lift $\tilde{\gamma}$ of $\gamma\in Sp_4(\Z/N\Z)$. 
We denote this action by the same notation $F|[\gamma]_{\underline{k}}$. 
This action does not depend on the choice of lifts of $\gamma$. 
The diagonal subgroup of $Sp_4(\Z/N\Z)$ is isomorphic to $(\Z/N\Z)^\times\times (\Z/N\Z)^\times$ by sending 
$S_{a,b}:={\rm diag}(a^{-1},b^{-1},a,b)$ to 
$(a,b)$ and it also acts on $M_{\underline{k}}(\G(N))$, factoring through the action of $Sp_4(\Z/N\Z)$. Then we have the character decomposition   
\begin{equation}\label{character}
M_{\underline{k}}(\G(N),c_1,c_2)=\bigoplus_{\chi_1,\chi_2:(\Z/N\Z)^\times\lra\C^\times}M_{\underline{k}}(\G(N),c_1,c_2, \chi_1,\chi_2),
\end{equation}
where  
$$M_{\underline{k}}(\G(N),c_1,c_2, \chi_1,\chi_2)=\{F\in M_{\underline{k}}(\G(N),c_1,c_2)\ |\ F|[S_{a,1}]_{\underline{k}}=\chi_1(a)F \ {\rm and}\  F|[S_{a,a}]_{\underline{k}}=\chi_2(a)F \}.
$$
It is easy to see that the Hecke operators preserve $M_{\underline{k}}(\G(N),c_1,c_2, \chi_1,\chi_2)$ (cf. \cite{manni&top} for the 
holomorphic case). 
We should remark that in order that $M_{\underline{k}}(\G(N),c_1,c_2, \chi_1,\chi_2)\ne 0$,
 the weight $(k_1,k_2)$ has to satisfy the parity condition 
\begin{equation}\label{parity}
\chi_2(-1)=(-1)^{k_1+k_2}. 
\end{equation}
Throughout this paper, we assume this parity condition on $F$. Let 
$$F(Z)=\ds\sum_{T\in P(\Z)_{\ge 0}}A_F(T,Y)e^{\frac{2\pi\sqrt{-1}}{N}{\rm tr}(TX)}\in 
M_{\underline{k}}(\G(N),c_1,c_2, \chi_1,\chi_2),\quad Z=X+Y \sqrt{-1},
$$ 
be an 
eigenform for all $T(p^i),\ p\nmid N,\ i\in\N$ with eigenvalues $\lambda(p^i)$, i.e.,
$$
T(p^i)F=\lambda(p^i)F.
$$
We next study the relation between $\lambda(p^i)$ and $A_F(T,Y)$. 
For a non-negative integer $\beta$, let $R(p^\beta)$ be the set of matrices 
$\begin{pmatrix}
u_1& u_2 \\
u_3 & u_4
\end{pmatrix}$ of $\G^{1}(N):=\{g\in SL_2(\Z)\ |\ g\equiv I_2\ {\rm mod}\ N \}$ whose first rows $(u_1,u_2)$ run 
over a complete set of representatives modulo the equivalence relation: 
$$(u_1,u_2)\sim (u'_1,u'_2) \Longleftrightarrow [u_1:u_2]=[u'_1:u'_2]\ {\rm in}\ \mathbb{P}^1(\Z/p^\beta\Z).
$$
Let $T(p^i)F=\ds\sum_{T\in P(\Z)_{\ge 0}}A_F(p^i;T,Y)e^{\frac{2\pi\sqrt{-1}}{N}{\rm tr}(TX)}$. For simplicity, we write 
$\rho_j={\rm Sym}^{j}{\rm St}_2$ for $j\ge 0$ and 
$UT{}^tU=\begin{pmatrix}
a_U& \frac{b_U}{2} \\
\frac{b_U}{2}  & c_U
\end{pmatrix}$ for $T\in P(\Z)_{\ge 0}$ and $U\in R(p^\beta)$. Put $D_\beta=
\begin{pmatrix}
1& 0 \\
0  & p^\beta
\end{pmatrix}$.
By Proposition 3.1 of \cite{evdokimov}, and the calculations done at p.439-440 of \cite{evdokimov}, and Section 3.1 of \cite{hori},  we have 
\begin{eqnarray*}
\label{hecke-fourier} 
&{}& \lambda(p^i)A_F(T,Y)=A_F(p^i;T,Y):=\sum_{\alpha+\beta+\gamma=i\atop \alpha,\beta,\gamma\ge 0}
\chi_1(p^\beta)\chi_2(p^\gamma)p^{\beta(k_2-2)+\gamma(k_1+k_2-3)}\times \\
&{}& \sum_{U\in R(p^\beta)\atop {a_U\equiv 0\ {\rm mod}\ p^{\beta+\gamma}\atop 
b_U\equiv c_U\equiv 0\ {\rm mod}\ p^{\gamma}}}
\rho_{k_1-k_2}(\begin{pmatrix}
1& 0 \\
0  & p^\beta
\end{pmatrix} U)^{-1}
A_F\left(p^\alpha \begin{pmatrix}
a_Up^{-\beta-\gamma}& \frac{b_Up^{-\gamma}}{2} \\
\frac{b_Up^{-\gamma}}{2}  & c_Up^{\beta-\gamma}
\end{pmatrix}, 
p^{\delta-2\alpha}D^{-1}_\beta {}^tU^{-1} YU^{-1}D^{-1}_\beta
\right). \nonumber
\end{eqnarray*}
\subsection{Adelic forms}\label{adelic}
%In this section we refer to \cite{borel&jacquet} and \cite{taylor-thesis}. 
Let $\A$ be the adele ring of $\Q$ and $\A_f=\widehat{\Z}\otimes_\Z\Q$ the finite adele of $\Q$.  
For a positive integer $N$ and a parabolic subgroup $R\in \{B,P,Q\}$, let $K_R(N)$ be the group consisting of the elements 
$g\in GSp_4(\widehat{\Z})$ such that $(g\ {\rm mod}\ N)\in 
R(\Z/N\Z)$. It is easy to see that $K_R(N)\cap Sp_4(\Q)=\G_R(N)$ and $\nu(K_R(N))=\widehat{\Z}$. 
%We denote by $K^{{\rm para}}(N)$ the adelic version of the paramodular group of level $N$ defined in \cite{r&s1}. 

Let $K(N)$ be the group consisting of the elements 
$g\in GSp_4(\widehat{\Z})$ such that $g\equiv I_4\ {\rm mod}\ N$. Then we see that  $\G(N)=Sp_4(\Q)\cap K(N)$ and $\nu(K(N))=1+N\widehat{\Z}$. 
Then it follows from the strong approximation theorem for $Sp_4$ that 
\begin{eqnarray}
 G(\A) &=& G(\Q)G(\R)^+K_R(N)=G(\Q)Z_G(\R)^+Sp_4(\R)K_R(N) \label{sat1} \\
 G(\A) &=& \coprod_{1\le a < N \atop (a,N)=1}G(\Q)G(\R)^+ d_a K(N)=\coprod_{1\le a < N \atop (a,N)=1}G(\Q)Z_G(\R)^+Sp_4(\R)d_a K(N) \label{sat2},
\end{eqnarray}
where $d_a$ is the diagonal matrix such that $(d_a)_p={\rm diag}(a,a,1,1)$ if $p|N$, $(d_a)_p=I_4$ otherwise. 

Let $I:=I_2\sqrt{-1}\in\mathcal{H}_2$ and $U(2)={\rm Stab}_{Sp_4(\R)}(I)$. 
Let $\mathfrak{g}_{0,\C}$ be the complexification of $\mathfrak{g}_{0}={\rm Lie}Sp_4(\R)$. 
We denote by $\mathcal{Z}$ the center of universal enveloping algebra of $\mathfrak{g}_{0,\C}$. 
Under the natural map $Sp_4(\R)\lra \mathcal{H}_2$, 
$g\mapsto g(I)$, $\mathcal{Z}\simeq \C[\Delta_1,\Delta_2]$. 
Choose $\widetilde{\Delta}_i\in \mathcal{Z}$ as in (\ref{generators}) in Section \ref{inf} and fix $\Delta_i$ which 
corresponds to $\widetilde{\Delta}_i$ under this map for $i=1,2$. 
For a function $\phi:GSp_4(\Q)\backslash GSp_4(\A)\lra V_{\underline{k}}(\C)$ and $D\in \mathfrak{g}_0$, 
we first define 
$$D \phi(g):=\lim_{t\to 0}\frac{d}{dt}\phi(g\exp(tD)),
$$
and extend this action linearly on $\mathfrak{g}_{0,\C}$.  
For any open compact subgroup $U$ of $GSp_4(\widehat{\Z})$ and complex numbers $c_1,c_2$, 
we let $\mathcal{A}_{\underline{k}}(U,c_1,c_2)^\circ$ denote the subspace of functions 
$\phi:GSp_4(\Q)\backslash GSp_4(\A)\lra V_{\underline{k}}(\C)$ such that 

\begin{enumerate}
\item $\phi(gu u_\infty)=\lambda_{\underline{k}}(J(u_\infty,I)^{-1})\phi(g)$ for all $g\in G(\A)$, $u\in U$, and $u_\infty\in U(2)$, 
\item for $h\in G(\A)$, the function 
$$\phi_h:\mathcal{H}_2\lra V_{\underline{k}}(\C),\ \phi_h(Z)=\phi_h(g_\infty I):=\lambda_{\underline{k}}(J(g_\infty,I))\phi(hg_\infty)$$
is a $C^\infty$ function where $Z=g_\infty I,\ g_\infty\in Sp_4(\R)$ (note that this definition is independent of the choice of $g_\infty$),
\item $\widetilde{\Delta}_i\phi=c_i \phi$ for $i=1,2$, 
\item for $g\in G(\A)$, $\ds\int_{N_R(\Q)\backslash N_R(\A)}\phi(ng)dn=0$ for any parabolic $\Q$-subgroup $R$ and $dn$ 
is the Haar measure on $N_R(\Q)\backslash N_R(\A)$. 
\end{enumerate}

We define similarly $\mathcal{A}_{\underline{k}}(U,c_1,c_2)$ by omitting the last condition (4). 

Let $\G(N)_a:=Sp_4(\Q)\cap d_a^{-1}K(N)d_a$. Note that $\G(N)_a=\G(N)$ for each $a$. 
Then we have the isomorphism 
\begin{equation}\label{isom}
\mathcal{A}_{\underline{k}}(K(N),c_1,c_2)\stackrel{\sim}{\lra} \bigoplus_{1\le a < N\atop (a,N)=1}
M_{\underline{k}}(\G(N)_a,c_1,c_2), \quad \phi\mapsto (\phi_{d_a}).
\end{equation}
The inverse of this isomorphism is given as follows: Let $F=(F_a)$ be an element of RHS which is a system of real analytic Siegel modular forms such that $F_a\in M_{\underline{k}}(\Gamma(N)_a,c_1,c_2)$ for each $a$. For each $g\in G(\A)$, there exists a unique $d_a$ such that
$g=rz_\infty d_a g_\infty k$ with $r\in G(\Q)$, $z_\infty\in Z_G(\R)^+$, $g_\infty\in Sp_4(\R)$, and $k\in K(N)$. Then we define the function  
$$\phi_{F}(g)=\lambda_{\underline{k}}(J(g_\infty,I))^{-1}F_a(g_\infty I).$$
This gives rise to the inverse of the above isomorphism. 
We also have the isomorphism 
$$\mathcal{A}_{\underline{k}}(K(N),c_1,c_2)^\circ\simeq \bigoplus_{1\le a < N\atop (a,N)=1}S_{\underline{k}}(\G(N)_a,c_1,c_2),
$$ 
as well (cf. \cite{borel&jacquet} for checking the cuspidality). 

Now we restrict the isomorphism (\ref{isom}) to a subspace, using the character decomposition (\ref{character}). 
Given two Dirichlet characters $\chi_i : (\Z/N\Z)^\times \lra \C^\times$, $i = 1, 2$, associate 
the characters $\chi'_i:\A_f^{\times}\lra \C^\times$ by the natural map $\A_f^\times\lra \widehat
\A_f^\times/\Q_{>0}={\widehat{\Bbb Z}}^\times\lra (\Z/N\Z)^\times$. 

Define
$\widetilde{\chi} : T(\A_f) \lra \C^\times$ by 
$$\tilde\chi'(\diag(*,*, c, d) =\chi'_1(d^{-1}c)\chi'_2(d).$$
Choose $F = (F_a)$ from RHS of (\ref{isom}) which satisfies 
$F|[S_{z,z}]_{\underline{k}} =(F_a|[S_{z,z}]_{\underline{k}}) = (\chi_2(z)F_a) = \chi_2(z)F$ and
$F|[S_{z,1}]_{\underline{k}} = \chi_1(z)F$. If we write $g\in G(\A)$ as $g = rz_\infty d_a g_\infty k \in G(\A)$ and take $z_f \in T(\A_f )$, then define the automorphic function attached to $F$ by
$$
\phi_F(g z_f) = \lambda_{\underline{k}}(J(g, I))^{-1} F_a(g_\infty I)\tilde\chi(z_f).
$$
Then this gives rise to the isomorphism of the subspaces 
\begin{equation*}
\mathcal{A}_{\underline{k}}(K(N),c_1,c_2, \tilde\chi)\stackrel{\sim}{\lra} \bigoplus_{1\le a < N\atop (a,N)=1} 
M_{\underline{k}}(\G(N)_a,c_1,c_2, \chi_1,\chi_2).
\end{equation*}

We now compute the actions of $\diag(1,1,-1,-1)\in GSp_4(\R)$ on $\phi_F\in \mathcal{A}_{\underline{k}}(K(N),c_1,c_2, \tilde\chi)$ as follows:
Let $h=(\diag(1, 1,-1,-1), I_{\A_f})=\diag(1,1,-1,-1)(I_{GSp_4(\R)}, (\diag(1,1,-1,-1))_p)\in G(\Q)(GSp_4(\R)\times GSp_4(\A_f))$,
where $I_{\A_f}$ (resp. $I_{GSp_4(\R)}$) is the identity element.
Then we have
\begin{eqnarray*}
\phi_F(gh) =\phi_F((r\cdot\diag(1, 1,-1,-1))z_\infty d_a g_\infty k \cdot \diag(1,1,-1,-1))_p)) \\
= \lambda_{\underline{k}}(J(g, I))^{-1} F_a(g_\infty I)\tilde\chi(\diag(1,1,-1,-1))_p) =\chi_2(-1) \phi_F(g),
\end{eqnarray*}
since we have assumed $\chi_2(-1) = (-1)^{k_1+k_2}$. Hence in the case of $(k_1,k_2)=(2,1)$, we have
\begin{equation}\label{sgn}
\phi_F(\diag(1, 1,-1,-1), I_{\A_f})=-1.
\end{equation}

\begin{remark}
Note that $M_{\underline{k}}(\Gamma(N),c_1,c_2)$ is embedded diagonally into 
$\ds\bigoplus_{1\le a < N\atop (a,N)=1} M_{\underline{k}}(\G(N)_a,c_1,c_2)$. 
So given a cusp form $F\in M_{\underline{k}}(\Gamma(N),c_1,c_2)$, we obtain $\phi_F\in \mathcal{A}_{\underline{k}}(K(N),c_1,c_2)$ which under the isomorphism (\ref{isom}), corresponds to $(F,...,F)$, and $\phi_F$ gives rise to a cuspidal representation $\pi_F$.
Conversely, given a cuspidal representation $\pi$ of $GSp_4/\Q$, there exists $N>0$ and $\phi\in \mathcal{A}_{\underline{k}}(K(N),c_1,c_2)$
which spans $\pi$. Under the isomorphism (\ref{isom}), $\phi$ corresponds to $(F_a)_{1\le a < N\atop (a,N)=1}$. 
For any $a$, let $\pi_{F_a}$ be the cuspidal representation associated to $F_a$. Then $\pi$ and $\pi_{F_a}$ have the same Hecke eigenvalues for $p\nmid N$, and hence they are in the same $L$-packet.
\end{remark}

We now study the Hecke operators on $\mathcal{A}_{\underline{k}}(K(N),c_1,c_2)$ and its relation to classical Hecke operators. 
Let $\phi$ be an element of  $\mathcal{A}_{\underline{k}}(K(N),c_1,c_2)$ and 
$F=(F_a)_a$ be the corresponding element of RHS via the above isomorphism (\ref{isom}). 
For any prime $p\nmid N$ and  $\alpha \in G(\Q)\cap T(\Q_p)$, define the Hecke action with respect to $\alpha$ 
$$\widetilde{T}_\alpha \phi(g):= \ds\int_{G(\A_f)}([K(N)_p\alpha K(N)_p]\otimes 1_{K(N)^p})\phi(g g_f) dg_f,
$$
where $dg_f$ is the Haar measure on $G(\A_f)$ so that vol$(K)=1$. Here $K(N)_p$ is the $p$-component of $K(N)$ and  
 $K(N)^p$ is the subgroup of $K(N)$ consisting of trivial $p$-component. 
 
Then by using (\ref{sat2}), we can easily see that 
 \begin{equation}\label{local-global}
T_\alpha F(Z)=\nu(\alpha)^{\frac{k_1+k_2}{2}-3}\widetilde{T}_{\alpha^{-1}} \phi(g),
\end{equation}
where $g=rz_\infty g_a g_\infty k$ as above and $Z=g_\infty I$ (cf. Section 8 of \cite{m&y}). 
From this relation, up to the factor of $\nu(\alpha)^{\frac{k_1+k_2}{2}-3}$, the isomorphism (\ref{isom}) preserves Hecke eigenforms in both sides. 

Conversely, let $F\in S_{\underline{k}}(\G(N),c_1,c_2)$ 
be a Siegel cusp form which is a Hecke eigenform. 
Then it is easy to see that $(F)_a$ is an eigenform in $\ds\bigoplus_{1\le a < N\atop (a,N)=1}S_{\underline{k}}(\G(N)_a,c_1,c_2)$. Hence we have the 
Hecke eigenform in $\mathcal{A}_{\underline{k}}(K(N),c_1,c_2)$ corresponding to $F$. 

For an open compact subgroup $U\subset G(\widehat{\Z})$, we say $U$ is of level $N$ if $N$ is the minimum positive integer so that 
$U$ contains $K(N)$. For such $U$ of level $N$, consider $\mathcal{A}_{\underline{k}}(U,c_1,c_2)$.
It is easy to generalize the above discussion to $\mathcal{A}_{\underline{k}}(U,c_1,c_2)$. We omit the details. 

The group $G(\A)$ acts on  $\ds\varinjlim_{U}\mathcal{A}_{\underline{k}}(U,c_1,c_2)$  
(also on  $\ds\varinjlim_{U}\mathcal{A}_{\underline{k}}(U,c_1,c_2)^\circ$) by 
right translation: 
$$(h\cdot \phi)(g):=\phi(gh), \ {\rm for}\ g,h\in G(\A).
$$
One would like to have scalar-valued functions lying in the usual $L^2(G(\Q)\backslash G(\A))$ space. Let
 $l:V_{\underline{k}}(\C)\lra \C$ be any linear functional. Define $\tilde\phi(g)=l(\phi(g))$. Since we consider all right translates of $\tilde\phi$, the choice of $l$ is irrelevant. For $\phi\in \mathcal{A}_{\underline{k}}(U,c_1,c_2)$ which is an eigenform for all $T_\alpha,\ \alpha\in G(\Q)\cap T(\Q_p)$ and $p\nmid N$, 
we denote by $\pi_\phi$, an irreducible direct summand of the representation of $G(\A)$ 
generated by $g\cdot \tilde\phi$ for $g\in G(\A)$ in $L^2(G(\Q)\backslash G(\A))$. Then $\pi_\phi$ is an automorphic representation in the sense of \cite{borel&jacquet}. 
Note that if the multiplicity one holds, $\pi_{\phi}$ is the irreducible representation generated by $g\cdot \tilde\phi$.
Further if $\phi \in 
\mathcal{A}_{\underline{k}}(U,c_1,c_2)^0$, then we see that $\pi_\phi$ is a cuspidal automorphic representation.   

Finally we remark on some compatibility related to the compact subgroups $K_\ast(N),\ \ast\in \{B,P,Q\}$. 
The obvious inclusions $M_{\underline{k}}(\G_\ast(N),c_1,c_2)\subset M_{\underline{k}}(\G_B(N),c_1,c_2) 
\subset M_{\underline{k}}(\G(N),c_1,c_2),\ \ast\in\{P,Q\}$ preserve the Hecke actions outside $N$. 
By (\ref{sat1}), one has $\mathcal{A}_{\underline{k}}(K_\ast(N),c_1,c_2)\simeq M_{\underline{k}}(\G_\ast(N),c_1,c_2)$. 
Then we have the following commutative diagram which preserves the Hecke actions outside $N$: 
$$
\begin{CD}
\mathcal{A}_{\underline{k}}(K_\ast(N),c_1,c_2) @> >> M_{\underline{k}}(\G_\ast(N),c_1,c_2)\\
@VVV @VVV \\
\mathcal{A}_{\underline{k}}(K(N),c_1,c_2) @> >> \bigoplus_{1\le a < N\atop (a,N)=1}M_{\underline{k}}(\G(N)_a,c_1,c_2).
\end{CD}
$$
Here the left vertical arrow is the natural inclusion and the right vertical arrow is the diagonal embedding 
(recall that $\G(N)_a=\G(N)$). 

\subsection{Conjectural existence of the rationality}\label{rationality}\label{rat}
In this section, let $\underline{k}=(2,1)$, $c_1=-\frac 5{12}$, $c_2=0$, and 
discuss a conjecture on the existence of some rational structure on $S_{(2,1)}(\G(N),-\frac 5{12},0)$.

Let $\mathbb{T}_\Q^{\text{univ}}$ be the Hecke algebra over $\Q$ which is generated by $T_{1,p}, T_{2,p}, \widetilde{S}_{p,1}$, and $\widetilde{S}_{p,p}$ for $p\nmid N$. 
We denote by $\mathbb{T}_\Q$, its image in ${\rm End}_\C(S_{(2,1)}(\G(N),-\frac 5{12},0))$. For any $\Q$-algebra $R$, put $\mathbb{T}_R=
\mathbb{T}_\Q\otimes_\Q R$. 

Let $F_1,\ldots,F_r$ be an orthonormal basis of $S_{(2,1)}(\G(N),-\frac 5{12},0)$ which consists of Hecke eigenforms. 

For $T\in T_\Q$ and $F_i$ for each $i=1,...,r$, denote by $a_T(F_i)$, the eigenvalue of $T$ for $F_i$. 
For $F=\ds\sum_{i=1}^r x_i F_i,\ x_i\in \C$, we define the map 
$$\psi: S_{(2,1)}(\G(N),-\frac 5{12},0)\lra {\rm Hom}(\mathbb{T}_\Q,\C),\quad
F\mapsto [T\mapsto \sum_{i=1}^r x_i a_T(F_i)]
$$ 
which depends on the choice of the basis $F_1,\ldots,F_r$.

We make the following rationality assumption for $F$:
\[
\langle TF\ |\ T\in \mathbb{T}_\Q \rangle_\C=\langle TF\ |\ T\in \mathbb{T}_\Q \rangle_\C\cap \C\otimes_\Q \psi^{-1}({\rm Hom}(\mathbb{T}_\Q,\Q)). \tag{\rm Rat}
\]

We denote by $\mathbb{T}_{\Q,F}$, the image of $\mathbb{T}_\Q$ in End$_\C(\langle TF\ |\ T\in \mathbb{T}_\Q \rangle_\C)$.

\begin{prop}\label{conj}Let $F\in  S_{(2,1)}(\G(N),-\frac 5{12},0)$. 
Assume (Rat) for $F$. Then the Hecke field $\Q_F$ is a finite 
extension over $\Q$. Furthermore, for any $\tau:\Q_F\hookrightarrow \C$, there exists ${}^\tau F\in  S_{(2,1)}(\G(N),-\frac 5{12},0)$ 
such that $T({}^\tau F)=\tau(a_T(F))({}^\tau F)$ for any $T\in \mathbb{T}_\Q$. 
\end{prop}
\begin{proof}By (Rat), there exists a basis $G_1,\ldots,G_r$ of $\langle TF\ |\ T\in \mathbb{T}_\Q \rangle_\C$ which gives a rational representation 
$\mathbb{T}_{\Q,F}\hookrightarrow {\rm End}_\Q(\langle G_1,\ldots, G_r \rangle_\Q)=M_r(\Q)$. 
This means that $\Q_F$ is a finite extension of $\Q$. Any Hecke eigenform $F$ in $\langle TF\ |\ T\in \mathbb{T}_\Q \rangle_\C$ can be written as 
$F=\ds\sum_{i=1}^r x_i G_i$ where $x_i$ is an element of a conjugate field of $\Q_F$ in $\C$. Hence 
we may set ${}^\tau F=\ds\sum_{i=1}^r \tau(x_i) G_i$.  
\end{proof}

\begin{remark} The rationality assumption holds for any vector-valued holomorphic Siegel cusp forms of arbitrary weight 
(cf. Lemma 2.1 of \cite{taylor-thesis}). Contrary to the holomorphic case, there is no known general result for (Rat) in 
our situation due to the lack of algebraic geometric structures. 
\end{remark} 

\section{Infinity type of the associated automorphic representation of $GSp_4$}\label{inf}

Let $F$ be a Hecke eigenform in $S_{(2,1)}(\G(N),-\frac 5{12},0)$ with the associated 
cuspidal representation $\pi_F=\pi_{\infty}\otimes \otimes_p' \pi_p$ of $GSp_4(\A)$. 
Let $\phi_F$ be the function on the adele group $GSp_4(\Bbb A)$ attached to 
$F$ as in section 3.3, and let $\tilde\phi_F=l(\phi_F)$ for any linear functional $l:V_{\underline{k}}(\C)\lra \C$.
 All the right translates of $\tilde\phi_F$ span $\pi_F$. Then $\tilde\phi_F=v_\infty\otimes \otimes_p' v_p$ with $v_p\in \pi_p$ and $v_\infty\in\pi_\infty$, and $v_\infty$ inherits the analytic properties from 
$\tilde\phi_F$. For simplicity, let us put $G=GSp_4(\R)$ and $B=B(\R)$ only in this section. 

In this section, we determine $\pi_\infty$.
We will use the notations from \cite{miy-oda}. It is easy to modify the result of \cite{miy-oda} for $G$.
For simplicity, let $K=K_{\infty}:={\rm Stab}_G(I_2\sqrt{-1})$.
Its identity component $K_0$ has index 2 in $K$, and it is a maximal compact subgroup of $Sp_4(\Bbb R)$, and we have the isomorphism 
$u: K_0\simeq U(2)$ via 
$u: k=\begin{pmatrix} A&B\\-B&A\end{pmatrix}\longmapsto u(k)=A+\sqrt{-1}B$.

We review principal series representations and their $K$-types from \cite{miy-oda}: Let ${\rm Ind}_B^G\, \chi$ be the principal series representation
which is the space $V$ of all $C^\infty$-functions $f: GSp_4(\Bbb R)\longrightarrow \Bbb C$ satisfying
$$f(tug)=\chi(t)|t_0|^{-\frac 32}|t_1|^2|t_2| f(g), \quad t=\diag(t_1,t_2,t_0 t_1^{-1}, t_0 t_2^{-1}).
$$
Here $B=TU$, $t\in T$, $u\in U$, and $\chi(t)=\prod_{i=0}^2 \epsilon_i(\frac {t_i}{|t_i|}) |t_i|^{s_i}$ with $\epsilon_i=1$ or $\sgn$, and $s=(s_0,s_1,s_2)\in \Bbb C^3$. We write $\chi=\chi(\epsilon_1|\, |^{s_1}, \epsilon_2 |\,|^{s_2}, \epsilon_0 |\, |^{s_0})$. 
Note that the infinity component of the central character is 
$\epsilon_{\infty}=\epsilon_1\epsilon_2\epsilon_0^2 |\,|^{s_1+s_2+2s_0}=\epsilon_1\epsilon_2 |\,|^{s_1+s_2+2s_0}$.
Since $\epsilon$ is trivial on $Z_G(\Bbb R)^+$ (It implies that $\epsilon$ is unitary and of finite order), 
$\epsilon(a)=|a|^{s_1+s_2+2s_0}=1$ for $aI_4\in Z_G(\Bbb R)^+$. Hence
$s_1+s_2+2s_0=0$. From (\ref{sgn}), we see that $\epsilon_0=\sgn$. 

If $(\rho,W)$ is a representation of $K$, and $v\in W$ has weight $(l,l')$, then $\rho(\diag(1,1,-1,-1))v$ has weight $(-l',-l)$. 
So if we let $\tau_{l,l'}$ be the representation of $K_0$ with dominant weight $(l,l')$, the weight structure of an irreducible representation of $K$ combines that of $\tau_{l,l'}$ and $\tau_{-l',-l}$ for some pair $(l,l')$. 

Now let $V_K$ be the subspace of $K$-finite vectors in the representation space $V$. Then
$$V_K=\bigoplus_{\lambda} m(\lambda) \tilde\tau_{\lambda},
$$
where $m(\lambda)$ is the multiplicity and 
$\tilde\tau_{\lambda}=\begin{cases} \tau_{l,l'}\oplus \tau_{-l',-l}, &\text{if $\lambda=(l,l')$, $l>0$, $l'\ne -l$,}\\
\tau_{l,-l}, &\text{otherwise.}\end{cases}
$

Since $\phi_F$ has weight $(2,1)$, $\tau_{2,1}$ occurs in $V_K$. Hence by \cite{miy-oda}, Proposition 3.2, $\epsilon_1\epsilon_2=\sgn$. Without loss of generality, we can assume that $\epsilon_1=1$ and $\epsilon_2=\sgn$. 

Let $H_1 = \diag(1,0,-1,0)$; $H_2=\diag(0,1,0,-1)$; $E_{e_1-e_2}$ is the matrix with 1 at $(1,2)$-entry, $-1$ at $(4,3)$-entry, and zero everywhere else;
$E_{e_1+e_2}$ is the matrix with 1 at $(1,4)$-entry and $(2,3)$-entry, and zero everywhere else;
$E_{2e_1}$ is the matrix with 1 at $(1,3)$-entry, and zero everywhere else;
$E_{2e_2}$ is the matrix with 1 at $(2,4)$-entry, and zero everywhere else. 
Let $E_{-\ast}:={}^t E_\ast$ for $\ast\in \{2e_1,2e_2,e_1+e_2,e_1-e_2\}$.
Let
$$M=\begin{pmatrix} H_1&E_{e_1-e_2}&2E_{2e_1}&E_{e_1+e_2}\\ E_{-e_1+e_2}&H_2&E_{e_1+e_2}&2E_{2e_2}\\2E_{-2e_1}&E_{-e_1-e_2}&-H_1&-E_{-e_1+e_2}\\E_{-e_1-e_2}&2E_{-2e_2}&-E_{e_1-e_2}&-H_2\end{pmatrix}.
$$
Then 
\begin{equation}
\label{generators} \widetilde{\Delta}_1=\frac 1{12} {\rm trace}(M^2)/2\ \ {\rm and}\ \widetilde{\Delta}_2=\det(M)
\end{equation}
give rise to two generators of the center of the universal enveloping algebra.
Here $\widetilde{\Delta}_1$ is the usual Casimir element as in \cite{knapp1}.

Let $v_\infty$ be the highest weight vector in the $K$-type $(2,1)$ in $\pi_\infty$.
Then by \cite{miy-oda}, page 77, 
\begin{equation}\label{casimir}
\widetilde{\Delta}_1 v_{\infty}=\frac {s_1^2+s_2^2-5}{12} v_{\infty},\quad \widetilde{\Delta}_2 v_{\infty}=s_1^2s_2^2 v_{\infty}.
\end{equation}
Hence in order that $\Delta_1 F=-\frac 5{12} F$ and $\Delta_2 F=0$, one should have $s_1=s_2=0$. 
Therefore, $\pi_{\infty}$ is a subquotient of the induced representation
${\rm Ind}_B^G\, \chi,$
where $\chi=\chi(1,\sgn, \sgn)$. 
Now under the Weyl group action, for $\chi=\chi(1,\sgn,\sgn)$,
$$\{ w\chi | w\in W\}=\{\chi(1,\sgn,1), \chi(\sgn,1,1), \chi(1,\sgn,\sgn), \chi(\sgn,1,\sgn)\},
$$ 
and 
${\rm Ind}_B^G \chi$ and ${\rm Ind}_B^G w\chi$ are equivalent. In particular, 
${\rm Ind}_B^G\, \chi(1,\sgn,1)$ and ${\rm Ind}_B^G\, \chi(1,\sgn,\sgn)$ are equivalent.

\begin{lemma}\label{knapp-stein} The Knapp-Stein $R$-group of ${\rm Ind}_B^G \chi(1,\sgn,\sgn)$ is trivial. Hence it is irreducible, tempered and generic.
\end{lemma}
\begin{proof} For the definition of the $R$-group, see \cite{knapp}.
If $\chi=\chi(1,\sgn,\sgn)$, we can see easily that $W_{\chi}=\{1, s_1s_2s_1\}$.
Hence the $R$-group is trivial. 
\end{proof}

Hence by the above lemma, $\pi_\infty\simeq {\rm Ind}_B^G\, \chi(1,\sgn,\sgn)$.
By \cite{mor}, page 414, the Langlands parameter of ${\rm Ind}_B^G \chi(1,\sgn,\sgn)$ is
$$\phi: W_{\Bbb R}\lra GSp_4(\C); \quad \phi(z)=I_4,\quad \phi(j)=\diag(1,-1,-1,1),
$$
and the Langlands parameter of ${\rm Ind}_B^G \chi(1,\sgn,1)$ is
$$\phi: W_{\Bbb R}\lra GSp_4(\C); \quad \phi(z)=I_4,\quad \phi(j)=\diag(-1,1,1,-1).
$$
In fact, they are conjugate in $GSp_4(\C)$, since $s_1 \diag(-1,1,1,-1) s_1=\diag(1,-1,-1,1)$, where $s_1$ is the simple reflection with respect to the short root, i.e.,
$s_1=\begin{pmatrix}
0 & 1 & 0 & 0\\
1 & 0 & 0 & 0\\
0 & 0 & 0 & 1\\
0 & 0 & 1 & 0
\end{pmatrix}.
$

We summarize our results as follows:

\begin{theorem} \label{infinity}
Let $F$ be a real analytic Siegel cusp form of weight $(2,1)$, level $N$ with the central character $\varepsilon$ such that $F$ has the eigenvalues $-\frac 5{12}$ and $0$ for the generators $\Delta_1$ and $\Delta_2$.  
Let $\pi_F=\pi_{\infty}\otimes \otimes_p' \pi_p$ be the associated cuspidal representation of $GSp_4(\Bbb A)$. 
Then $\pi_{\infty}\simeq {\rm Ind}_B^G\, \chi(1,\sgn,\sgn)$. 
The Langlands parameter of $\pi_{\infty}$ is given by
$$\phi: W_{\Bbb R}\longrightarrow GSp_4(\Bbb C),\quad \text{$\phi(z)=I_4$ for $z\in\Bbb C$},\quad
\phi(j)=\diag(1,-1,-1,1).
$$
\end{theorem}

Note that $\pi_\infty$ is generic and tempered. It is a totally degenerate limit of discrete series in the sense of \cite{CK}. It is denoted by $D_{(1,0)}[0]$ in \cite{mor}, page 414. which is a limit of large discrete series. 
Since ${\rm Ind}_B^G\, \chi(1,\sgn,\sgn)$ is irreducible, the $L$-packet of $\pi_\infty$ is a singleton.

\begin{remark}\label{(1,0)}
Note that the minimal $K$-types of $\pi_\infty$ are $(1,0)$ and $(0,-1)$. 
One can get the highest weight vector in the $K$-type $(2,1)$ from the one in the $K$-type $(1,0)$
by first applying the $\rho(\diag(1,1,-1,-1))$, and then by taking a differential operator. Since the highest weight vectors give rise to Siegel modular forms, we can describe this differential operator on Siegel modular forms explicitly as follows \cite{harris}:
Let 
$Z=X+Y \sqrt{-1}=\begin{pmatrix} z_{11}&z_{12}\\z_{12}&z_{22}\end{pmatrix}\in \mathcal{H}_2$.
Let $F$ be a $C^\infty$ Siegel modular form of weight $(0,-1)$, namely, the weight corresponding to 
${\rm St}_2\otimes \det^{-1}{\rm St}_2$. 
Define the differential operator $(\partial_{ij})_{1\le i\le j\le 2}$ on such $F$ by  
$$
2\sqrt{-1}\det(Z)Z^{-1}Y^{1/2}\Big(\frac{\partial}{\partial z_{ij}}\Big)_{ij}Y^{1/2}+ \left(\Big(\frac{\partial}{\partial z_{ij}}\Big)_{ij}+I_2 \right)\det(Y^{1/2})^{-1}Y^{1/2},
$$
where $\Big(\frac{\partial}{\partial z_{ij}}\Big)_{ij}=\begin{pmatrix}
\frac{\partial}{\partial z_{11}}& \frac{1}{2}\frac{\partial}{\partial z_{12}} \\
 \frac{1}{2}\frac{\partial}{\partial z_{12}} & \frac{\partial}{\partial z_{22}} 
\end{pmatrix}$. 
Then define the differential operator $D$ as 
$$D:=\det((\partial_{ij})_{1\le i\le j\le 2}).
$$
Then by a minor modification of the calculation in Section 6 of \cite{harris}, we have 
$$D:M_{(0,-1)}(\G(N),c_1,c_2)\lra M_{(2,1)}(\G(N),c_1,c_2),
$$
which commutes with the actions of Hecke operators outside $N$. 
Hence we have a map 
$$M_{(1,0)}(\G(N),c_1,c_2)\stackrel{F(Z)\mapsto \det(Y)F(Z)}{\lra} M_{(0,-1)}(\G(N),c_1,c_2)\stackrel{D}{\lra} M_{(2,1)}(\G(N),c_1,c_2)$$
which preserves eigenforms. 
Note that for $F\in M_{(1,0)}(\G(N),c_1,c_2)$, if $T(p^i)F=\alpha(p^i)F$, then one can easily see that 
$T(p^i)\det(Y)F=p^i\alpha(p^i)\det(Y)F$ which explains $L^N(s-1,F)=L^N(s,D(\det(Y)F))$ as in Remark \ref{weightchange}.
Since  $D(\det(Y)F)$ and $F$ give rise to the same automorphic 
representation, the image of such an $F$ under this map is nonvanishing if $F$ is a cusp form. 
\end{remark}

\begin{remark} 
The minimal $K$-type of ${\rm Ind}_B^G\, \chi(1,1,\epsilon_0)$, where $\epsilon_0=1$ or sgn, 
is $(0,0)$. The minimal $K$-types of ${\rm Ind}_B^G\, \chi(\sgn,\sgn,\epsilon_0)$, where $\epsilon_0=1$ or sgn, 
are $(1,1)$ and $(-1,-1)$. 
Also the Langlands parameter of ${\rm Ind}_B^G \chi(1,1,1)$ is
$$\phi: W_{\Bbb R}\lra GSp_4(\C); \quad \phi(z)=I_4,\quad \phi(j)=I_4.
$$
The Langlands parameter of ${\rm Ind}_B^G \chi(\sgn,\sgn,1)$ is
$$\phi: W_{\Bbb R}\lra GSp_4(\C); \quad \phi(z)=I_4,\quad \phi(j)=\diag(1,-1,1,-1).
$$
So real analytic Siegel cusp forms of weight $(0,0)$ and $(1,1)$ with the eigenvalues $-\frac 5{12}$ and $0$ for the two generators of the algebra of all $Sp_4(\R)$-invariant differential operators on the Siegel upper half plane, would correspond to symplectically even Artin representations.
\end{remark}

\section{Correspondence between automorphic representations of $GSp_4$ and $GL_4$}\label{Arthur}\label{arthur}

J. Arthur \cite{arthur} described the correspondence between automorphic representations of $GSp_4(\A)$ and $GL_4(\A)$, under the validity of stabilization of the twisted trace formula for $GSp_4$. We assume his result. In fact, we only need the transfer from the 
cuspidal representation $\pi_F$ of
$GSp_4(\A)$ to an automorphic representation $\Pi$ of $GL_4(\A)$ so that $L(s,\pi_F)=L(s,\Pi)$. 

We summarize his results on $L_{{\rm disc}}^2(G(F)\backslash G(\Bbb A),\chi)$
as follows. According to its transfer to $GL_4$, it is divided into 6 families of global $L$-packets:
\begin{enumerate}
\item stable, semisimple type: its transfer to $GL_4$ is a cuspidal automorphic representation of $GL_4(\A)$ which is not orthogonal type.

\item unstable, semisimple (Yoshida type): its transfer to $GL_4$ is an isobaric sum of two distinct cuspidal automorphic representations of 
$GL_2(\A)$ with the same central character $\chi$. This is called endoscopic type, 

\item stable, mixed (Soudry type): it is a CAP representation from Klingen parabolic subgroup with a cuspidal automorphic representation 
$\pi$ of $GL_2(\A)$ of orthogonal type such that $\omega_{\pi}^2=\chi$. In this case, $L(s, Ad(\pi)\otimes\eta)$ has a pole at $s=1$, where 
$\eta$ is determined by $\pi$, and $Ad(\pi)$ is the Gelbart-Jacquet lift to $GL_3$ \cite{GJ}. 
Its transfer to $GL_4$ is the residual automorphic representation which is the Langlands quotient of 
${\rm Ind}_{GL_2\times GL_2}^{GL_4}\, \pi |\det|^{\frac 12}\times \pi|\det|^{-\frac 12}$.

\item unstable, mixed (Saito-Kurokawa type): it is a CAP representation from Siegel parabolic subgroup with a cuspidal automorphic representation $\pi$ of $GL_2(\A)$ and a character $\lambda$ such that $\omega_{\pi}=\lambda^2=\chi$. Its transfer to $GL_4$ is the isobaric representation $\pi\boxplus \chi(\det_2)$, where $\chi(\det_2)$ is the quotient of ${\rm Ind}_B^{GL_2} \chi|\,|^{\frac 12}\otimes \chi|\, |^{-\frac 12}$.

\item unstable, almost unipotent (Howe-Piatetski-Shapiro type): it is a CAP representation from the Borel subgroup with two gr\"ossencharacters $\chi_1,\chi_2$ such that $\chi_1^2=\chi_2^2=\chi$. Its transfer to $GL_4$ is the isobaric representation $\chi_1(\det_2)\boxplus \chi_2(\det_2)$.

\item stable, almost unipotent (one dimensional type): Its transfer to $GL_4$ is $\lambda(\det_4)$ with $\lambda^4=\chi$, which is the Langlands quotient of 
${\rm Ind}_B^{GL_4} \lambda|\, |^{\frac 32}\otimes \lambda|\, |^{\frac 12}\otimes \lambda|\, |^{-\frac 12}\otimes \lambda|\, |^{-\frac 32}$.
\end{enumerate}

\begin{lemma}\label{CAP} Let $F$ be a Hecke eigenform in $S_{(2,1)}(\G(N),-\frac 5{12},0)$ with the associated 
cuspidal representation $\pi_F$ of $GSp_4(\A)$. Then $\pi_F$ is not a CAP representation.
\end{lemma}
\begin{proof} By Theorem 2.1 of \cite{ps0}, the central character of any CAP representation associated to Borel subgroup or Siegel parabolic subgroup is the square of a character. Hence
the parity condition (\ref{parity}) implies that $\pi_F$ is not a CAP representation associated to Borel subgroup or Siegel parabolic subgroup. 
If $\pi_F$ is a CAP associated to Klingen parabolic subgroup, $\pi_F$ comes from a theta lifting from $GO(2,T)_\Q$ where $T$ is 
an anisotropic quadratic form over $\Q$. Let $D_T$ be the discriminant of $T$ and let $K=\Q(\sqrt{D_T})$. 
Let $\chi_T:\A^\times \lra \C^\times$ be the quadratic character associated to $K/\Q$.  
Then by \cite{soudry}, Theorem A and the proof of Lemma 1.4 of \cite{soudry}, 
there exists a non-trivial map 
$$\pi_F\lra {\rm Ind}^{GSp_4(\A)}_{Q(\A)} \, \chi_T |\,|^{-1}_{\A}\rtimes (\sigma\otimes |det|_{\A}),
$$
where $\sigma$ is the automorphic induction to $GL_2(\A)$ from a unitary Hecke character $\varphi$ on $\A^\times_K$.
Since $\pi_F$ is irreducible, this map is injective. In particular, $\pi_\infty$ is a subquotient of the principal series 
$Ind_B^G\: \chi(|\,|^{-1}\epsilon_1, |\,|^{s_1}\epsilon_2, |\,|^{s_0}\epsilon_0)$ for some $s_0,s_1\in\Bbb C$, and $\epsilon_i=1$ or sgn.
However by Theorem \ref{infinity}, $\pi_\infty=\pi(1,\sgn,\sgn)$. This contradicts to Harish-Chandra's subquotient theorem (cf. \cite{Mu}, Theorem 2.1).
\end{proof}

Hence conjecturally $\pi_F$ falls into the case (1) or (2). So
\begin{theorem} (TR)\label{transfer} 
Let $\pi_F$ be as above. Then there exists an automorphic representation $\Pi$ of $GL_4(\A)$ which is either cuspidal or an isobaric sum of two distinct cuspidal automorphic representations of $GL_2$ 
such that $L(s,\Pi)=L(s,\pi_F)$.
\end{theorem}

\section{Application of Rankin-Selberg method}

\subsection{Spinor L-functions}
Let $F$ be a real analytic Siegel cusp form of weight $(k_1,k_2)$ so that $F|[S_{p,1}]=\chi_1(p)$ and $F|[S_{p,p}]=\chi_2(p)$ for 
$p\nmid N$. 
We assume that $F$ is a Hecke eigenform of $T_{i,p}$ for  $p\nmid N$ with eigenvalues $a_{i,p},\ i=1,2$, i.e.,
$a_{1,p}=\lambda(p)$ and 
$$\lambda(p)^2-\lambda(p^2)-p^{k-1}\chi_2(p)=pa_{2,p}+p^{k}(1+p^{-2})\chi_2(p),
$$
where $k=k_1+k_2-3$.
Then we define the partial L-function of $F$ by 
$L^{N}(s,F):=\ds\prod_{p{\not\hspace{0.7mm}|}N}L_p(s,F),$
$$
L_p(s,F)=(1-a_{1,p}T+\{pa_{2,p}+p^{k}(1+p^{-2})\chi_2(p)\}T^2-\chi_2(p)p^{k}a_{1,p}T^3+\chi_2(p)^2p^{2k}T^4)^{-1},
$$
where $T=p^{-s}$. 
Let $\pi_F$ be the cuspidal automorphic representation of $GSp_4$ associated to $F$, and 
let $L^N(s,\pi_F)=\ds\prod_{p\nmid N} L_p(s, \pi_{p})$ be the partial automorphic L-function of $\pi_F$. 
Then by definition, we have 
\begin{equation}\label{L}
L^N(s-\tfrac{k_1+k_2-3}{2},\pi_F)=L^N(s,F). 
\end{equation}
The L-function $L^{N}(s,F)$ converges in some half plane Re$(s)\gg 0$, and has a 
meromorphic continuation to the whole complex plane (\cite{langlands0}). By Arthur's conjecture (Theorem \ref{transfer}), it satisfies the desired functional equation and is entire. When $\pi_F$ is globally generic, Moriyama \cite{mor} proved the analytic properties of $L^N(s,F)$.

We now discuss the relation between $a_{i,p}$ and Satake parameters. 
Let $K(N)_p$ be the $p$-component of $K(N)$. Then $K(N)_p=GSp_4(\Z_p)$ for any $p\nmid N$. 
So if $p\nmid N$, $\pi_{p}$ is a spherical unitary representation which can be written as the irreducible quotient of 
Ind$^{G(\Q_p)}_{B(\Q_p)}\, \chi(\mu_1,\mu_2,\eta)$, where $\mu_1,\mu_2,\eta:\Q^\times_p\lra \C^\times$ are quasi-characters. 
Note that the central character of $\pi_{p}$ is $\varepsilon:=\mu_1\mu_2\eta^2$. 

Let 
$$\alpha_{0p}=\eta(p^{-1}),\quad \alpha_{1p}=\mu_1(p^{-1}),\quad \alpha_{2p}=\mu_2(p^{-1}).
$$
Then under the map
${}^L GSp_4=GSp_4(\Bbb C)\hookrightarrow GL_4(\Bbb C)$, 
the Satake parameter corresponding to $\pi_{p}$ is (see, for example, \cite{soudry}, p 95)
$${\rm diag}(\alpha_{0p}\alpha_{1p}\alpha_{2p},\, \alpha_{0p}\alpha_{1p},\, \alpha_{0p}\alpha_{2p},\, \alpha_{0p}).
$$
Then by using the adelic form of $F$ and the relation (\ref{local-global}), we can easily see that 
\begin{equation}\label{a1p}
a_{1,p}=p^{\frac{k_1+k_2}{2}-3}p^{\frac{3}{2}}(\alpha_{0p}\alpha_{1p}\alpha_{2p}+\alpha_{0p}\alpha_{1p}+\alpha_{0p}\alpha_{2p}+\alpha_{0p})
=p^{\frac{k_1+k_2-3}{2}}\alpha_{0p}(1+\alpha_{1p})(1+\alpha_{2p})
\end{equation}
and 
\begin{eqnarray}\label{a2p}
&{}& pa_{2,p}+(p^{k_1+k_2-5}+p^{k_1+k_2-3})\chi_2(p) \\
&=& p^{2(\frac{k_1+k_2-3}{2})}(\alpha_{0p}^2 \alpha_{1p}+\alpha_{0p}^2\alpha_{2p}+\alpha_{0p}^2\alpha_{1p}\alpha_{2p}+\alpha_{0p}^2\alpha_{1p}\alpha_{2p}+\alpha_{0p}^2\alpha_{1p}^2\alpha_{2p}+\alpha_{0p}^2\alpha_{1p}\alpha_{2p}^2) \nonumber \\
&=& p^{2(\frac{k_1+k_2-3}{2})}\varepsilon(p^{-1})(\alpha_{1p}+\alpha_{2p}+2+\alpha_{1p}^{-1}+\alpha_{2p}^{-1}). \nonumber
\end{eqnarray}
where $\varepsilon(p^{-1})=\chi_2(p)$. Here the factor $p^{\frac{3}{2}}$ in (\ref{a1p}) is the contribution from the value of  
$\delta^{-\frac{1}{2}}_B$ at 
${\rm diag}(1,1,p^{-1},p^{-1})$ where $\delta_B$ is the modulus character of $B$ defined by 
$\delta_B(h)=|a|^4_p|b|^2_p|c|^{-3}_p$ for 
$h={\rm diag}(a,b,ca^{-1},cb^{-1})u\in B(\Q_p)=T(\Q_p)U(\Q_p)$. 
A reason why the factor $\delta^{-\frac{1}{2}}_B({\rm diag}(1,1,p^{-1},p^{-1}))$ appears there is because 
the eigenvalues (namely, Satake parameters) of a non-zero spherical vector $(\pi_{p})^{GSp_4(\Z_p)}$ are usually computed via the Jacquet module 
with respect to $B$ (cf. Proposition 2.3 of \cite{m&y}).

\subsection{Rankin-Selberg $L$-functions}\label{ran-sel}
Let $F$ be a Hecke eigenform in $S_{(2,1)}(\G(N),-\frac 5{12},0)$ with the associated 
cuspidal representation $\pi_F$ of $GSp_4(\A)$.
Let $\Pi$ be the transfer of $\pi_F$ to $GL_4/\Bbb Q$ by our assumption (TR). Then by (\ref{L}),
$$L(s,\Pi)=L(s,\pi_F)=\sum_{n=1}^\infty \frac {a(n)}{n^s},
$$
where $a(p)=\lambda(p)$. Now consider the Rankin-Selberg $L$-function $L(s,\Pi\times\tilde\Pi)$ \cite{JS}. Note that either $\Pi$ is cuspidal or $\Pi=\pi_1\boxplus\pi_2$, where $\pi_i$'s are cuspidal representation of $GL_2/\Bbb Q$ with $\pi_1\not\simeq\pi_2$. Hence $L(s,\Pi\times\tilde\Pi)$ has a pole of order 1 at $s=1$ when $\Pi$ is cuspidal, and of order 2 when $\Pi=\pi_1\boxplus \pi_2$. Also $L(s,\Pi\times\tilde\Pi)$ has no zeros in a small neighborhood of $s=1$. Let $\{\beta_1,...,\beta_4\}$ be the Satake parameter of $\Pi_p$ so that $\lambda(p)=\beta_1+\cdots+\beta_4$. Then by Euler product expansion,
$\log L(s,\Pi)=\sum_{m=1}^\infty \sum_p \frac {b(p^m)}{m p^{ms}},$ where $b(p^m)=\beta_1^m+\cdots+\beta_4^m$.
We can show easily that
$$\log L(s,\Pi\times\tilde\Pi)=\sum_{m=1}^\infty\sum_p \frac {|b(p^m)|^2}{m p^{ms}}.
$$
So if $s>1$, $\log L(s,\Pi\times\tilde\Pi)\geq \sum_p \frac {|\lambda(p)|^2}{p^{s}}.$
Hence as $s\to 1^+$,
\begin{equation} \label{rankin}
\sum_p \frac {|\lambda(p)|^2}{p^s}\leq \begin{cases} \log \frac 1{s-1}+O(1), &\text{if $\Pi$ is cuspidal},\\
2\log \frac 1{s-1}+O(1), &\text{if $\Pi=\pi_1\boxplus \pi_2$}.
\end{cases}
\end{equation}

\begin{remark}
In fact, $\log L(s,\Pi\times\tilde\Pi)=\sum_p \frac {|\lambda(p)|^2}{p^s}+g(s)$ for a holomorphic function $g(s)$ near $s=1$.
We only need to show that $\sum_{m=2}^\infty \sum_p \frac {|b(p^m)|^2}{m p^{ms}}$ converges at $s=1$. 
By \cite{LRS}, $|b(p^m)|\leq 4 p^{\frac m2-\frac m{17}}$. Hence 
$\sum_p \frac {|b(p^m)|^2}{m p^{ms}}\ll \sum_p \frac {|b(p^m)|}{p^{ms-\frac m2+\frac m{17}}}$. By \cite{cho-kim}, if $m\geq 2$,
$\sum_{p\leq x} |b(p^m)|\ll x^{\frac m2}$. Hence by partial summation, for $m\geq 2$,
$$\sum_p \frac {|b(p^m)|}{p^{ms-\frac m2+\frac m{17}}}
\ll \int_2^\infty x^{\frac m2} x^{-ms+\frac m2-\frac m{17}-1}\, dx=2^{-ms+m-\frac m{17}}.
$$
Since $\sum_{m=2}^\infty 2^{-ms+m-\frac m{17}}$ converges at $s=1$, our result follows.
\end{remark}

Now let $\Pi_5$ be the transfer of $\pi_F$ to $GL_5/\Bbb Q$. It is obtained as follows: Let $\tau$ be an automorphic representation of 
$GL_6$ such that $\tau_p\simeq \wedge^2(\Pi_p)$ for $p\ne 2,3$ \cite{Kim}. Since $\Pi$ is the transfer of $\pi_F$, $L(s,\Pi,\wedge^2\otimes\varepsilon^{-1})$ has a pole at 
$s=1$ \cite{Kim-Sh1}.
Hence $\tau$ is an isobaric automorphic representation given by
$$\tau=(\Pi_5\otimes\varepsilon)\boxplus \varepsilon,
$$
where $\Pi_5$ is an automorphic representation of $GL_5$. It is easy to see that $\Pi_5$ is a weak transfer of $\pi_F$ to $GL_5$ corresponding to the $L$-group homomorphism $GSp_4(\Bbb C)\longrightarrow GL_5(\Bbb C)$, given by the second fundamental weight \cite{Kim1}.

The Satake parameter for $\wedge^2(\Pi)_p$ is
\begin{eqnarray*}
{\rm diag}(\alpha_{0p}^2 \alpha_{1p},\, \alpha_{0p}^2\alpha_{2p},\, \alpha_{0p}^2\alpha_{1p}\alpha_{2p},\, \alpha_{0p}^2\alpha_{1p}\alpha_{2p},\, \alpha_{0p}^2\alpha_{1p}^2\alpha_{2p},\, \alpha_{0p}^2\alpha_{1p}\alpha_{2p}^2) \\
={\rm diag}(\varepsilon(p^{-1})\alpha_{1p},\, \varepsilon(p^{-1})\alpha_{2p},\, \varepsilon(p^{-1}),\, \varepsilon(p^{-1})\alpha_{1p}^{-1},\, \varepsilon(p^{-1})\alpha_{2p}^{-1}).
\end{eqnarray*}

Hence the unramified factor of the standard $L$-function of $\Pi_5$ is given by
$$(1-\alpha_{1p}T)(1- \alpha_{2p}T)(1- T)(1- \alpha_{1p}^{-1}T)(1-\alpha_{2p}^{-1}T).
$$

Let 
$$L(s,\Pi_5)=\sum_{n=1}^\infty \frac {c(n)}{n^s}.
$$
Then $c(p)=\alpha_{1p}+\alpha_{2p}+1+\alpha_{1p}^{-1}+\alpha_{2p}^{-1}$. This is called the standard $L$-function of $F$.

Consider the Rankin-Selberg $L$-function $L(s,\Pi_5\times\tilde \Pi_5)$. Note that $L(s,\Pi_5\times\tilde\Pi_5)$ has a pole of at least order 1, and at most order 5 at $s=1$.
For $s>1$, $\log L(s,\Pi_5\times\tilde\Pi_5)\geq \sum_p \frac {|c(p)|^2}{p^s}$. Hence
$\sum_p \frac {|c(p)|^2}{p^s}\leq 5\log\frac 1{s-1}+O(1)$ as $s\to 1^+$.

Now from (\ref{a2p}),
\begin{eqnarray*}
\lambda(p)^2-\lambda(p^2)-\varepsilon(p^{-1})p^{-1} &=&
\alpha_{0p}^2 \alpha_{1p}+\alpha_{0p}^2\alpha_{2p}+\alpha_{0p}^2\alpha_{1p}\alpha_{2p}+\alpha_{0p}^2\alpha_{1p}\alpha_{2p}+\alpha_{0p}^2\alpha_{1p}^2\alpha_{2p}+\alpha_{0p}^2\alpha_{1p}\alpha_{2p}^2 \\
 &=& \varepsilon(p^{-1})(c(p)+1).
\end{eqnarray*}

Hence $\lambda(p^2)=\lambda(p)^2-\varepsilon(p^{-1}) p^{-1}-\varepsilon(p^{-1})(c(p)+1)$.

Let $L$ be the Galois closure of the Hecke field $\Bbb Q_F=\Bbb Q(\lambda(p),c(p),\varepsilon(p^{-1}), p\nmid N)$ 
and $\mathcal O_L$ be the ring of integers of 
$L$. Here $\Bbb Q_F=\Bbb Q(\lambda(p),\lambda(p^2),\varepsilon(p^{-1}), p\nmid N)$.
We make the integrality assumption:
\[
\text{For all $p\nmid N$,\quad $\lambda(p)$, $\lambda(p)^2-\lambda(p^2)-\epsilon(p^{-1})p^{-1}\in \mathcal O_L.$} \tag{\rm Int}
\]

It is equivalent to $\lambda(p), c(p)\in \mathcal O_L$ for all $p\nmid N$. Now we prove

\begin{prop}\label{density} Under (TR) and (Int), for any real $\eta>0$, there exists a set $X_{\eta}$ of rational primes such that 
${\rm den.sup}X_{\eta}\leq \eta$, and the set $\{(\lambda(p), \lambda(p^2))\, |\, p\notin X_{\eta}\}$ is a finite set, or equivalently, 
$\{ \text{Satake parameters at $p$}\, |\, p\notin X_{\eta}\}$ is finite.
\end{prop}
Here den.sup$X_{\eta}$ is defined by
$$\limsup_{s\to 1^+} \frac {\sum_{p\in X_{\eta}} p^{-s}}{\log \frac 1{s-1}}.
$$
We also define the Dirichlet density den$(X_{\eta})$ by 
$$\lim_{s\to 1^+} \frac {\sum_{p\in X_{\eta}} p^{-s}}{\log \frac 1{s-1}}.
$$
\begin{proof}
For $c>0$, consider two sets:
\begin{eqnarray*}
Y(c)&=&\{ \text{$a\in\mathcal O_L \,| \, |\sigma(a)|^2\leq c$ for any $\sigma\in {\rm Gal}(L/\Bbb Q)$}\},\\
X(c)&=&\{\text{$p\, |\,$ $\lambda(p)$ or $c(p)$ does not belong to $Y(c)$}\}.
\end{eqnarray*}

Note that since $\mathcal O_L$ is a lattice, $Y(c)$ is a finite set for any $c>0$. By the assumption (Int), 
$\lambda(p), c(p)\in \mathcal{O}_L$. So if $p\notin X(c)$, $\lambda(p), c(p)\in Y(c)$. 
Hence the set $\{ (\lambda(p), c(p))\, |\,\, p\notin X(c)\}$
is finite. Since $\lambda(p^2)=\lambda(p)^2-\varepsilon(p^{-1}) p^{-1}-\varepsilon(p^{-1})(c(p)+1)$, the set $\{ (\lambda(p), \lambda(p^2)) \,|\, p\notin X(c) \}$ is finite.

For each $\sigma\in {\rm Gal}(L/\Bbb Q)$, ${}^\sigma F$ is a cuspidal eigenform with $T(p)({}^\sigma F)=\sigma(\lambda(p))({}^\sigma F)$. 
Hence 

$$\sum_{\sigma}\sum_p \frac {|\sigma(\lambda(p)|^2}{p^s}\leq 2r\log\frac 1{s-1}+O(1),\quad \text{as $s\to 1^+$}
$$
where $r=[L:\Bbb Q]$. Also 
$$\sum_{\sigma}\sum_p \frac {|\sigma(c(p)|^2}{p^s}\leq 5r\log\frac 1{s-1}+O(1),\quad \text{as $s\to 1^+$}
$$

If $p\in X(c)$, $|\sigma_0(\lambda(p))|^2>c$ or $|\sigma_1(c(p))|^2>c$ for some $\sigma_0, \sigma_1\in {\rm Gal}(L/\Bbb Q)$.
Therefore
$$c\sum_{p\in X(c)} p^{-s}\leq \sum_{\sigma}\sum_p \frac {|\sigma(\lambda(p))|^2}{p^s}+\sum_{\sigma}\sum_p \frac {|\sigma(c(p))|^2}{p^s}\leq 7r\log\frac 1{s-1}+O(1), \quad \text{as $s\to 1^+$}.
$$

Hence, den.sup$X(c)\leq \frac {7r}c$. Take $c$ such that $c\geq \frac {7r}{\eta}$, and $X_{\eta}=X(c)$. This proves Proposition \ref{density}.
\end{proof}
\begin{remark}\label{(2,1)}
If $F$ is a Siegel cusp form of weight $(k_1,k_2)$, then by (\ref{L}),
$L(s-\tfrac{k_1+k_2-3}{2},\pi_F)=L(s,F)$. So
$$\sum_p \frac {|\lambda(p)|^2}{p^s}=a\log\frac 1{s-(k_1+k_2-2)}+O(1),\quad \text{as $s\to k_1+k_2-2$}
$$
where $a=1$ or 2. Hence only when $(k_1,k_2)=(2,1)$, we can use the above argument.
\end{remark}

\section{Conjecture on the existence of mod $\ell$ Galois representations}\label{mod-Galois}

In this section we formulate a conjecture on the existence of the mod $\ell$ Galois representations 
attached to a real analytic Siegel modular form $F$ of weight $(2,1)$, in analogy with holomorphic Siegel cusp forms.

Let $F\in S_{(2,1)}(\Gamma(N),-\frac 5{12},0)$ be a Hecke eigenform with eigenvalues $\lambda(p^i)$  for $T(p^i)$, $F|[S_{p,1}]=\chi_1(p)F$, and $F|[S_{p,p}]=\chi_2(p)F$  ($p\nmid  N$). 
Let $\pi_F=\pi_{\infty}\otimes \otimes_p' \pi_p$ be the cuspidal automorphic representation attached to $F$. 
Recall that $\pi_F$ is not a CAP representation (Lemma \ref{CAP}).

\begin{conjecture}\label{galois} Assume (Rat) and (Int) for $F$. 
Let $\ell$ be an odd prime which is coprime to $N$. 
Then for each finite place $\lambda$ of $\Q_F$ with the 
residue field $\F_\lambda$, there exists a continuous semi-simple representation 
$$\rho_{F,\lambda} : G_\Q\lra GSp_4(\bF_\lambda),
$$
which is unramified outside of $\ell N$ 
so that 
$$\det(I_4-\rho_{F,\lambda}({\rm Frob}_p) T)\equiv 
1-\lambda(p)T+\{\lambda(p)^2-\lambda(p^2)-p^{-1}\chi_2(p)\}T^2-\chi_2(p)\lambda(p)T^3+\chi_2(p)^2T^4\, {\rm mod}\,\lambda,
$$
for any $p\nmid \ell N$. Furthermore, $\rho_{F,\lambda}$ is symplectically odd, i.e. $\rho_{F,\lambda}(c)$ has eigenvalues $1,1,-1,-1$ and
$\rho_{F,\lambda}(c)\stackrel{\tiny{GSp_4}}{\sim }{\rm diag}(1,-1,-1,1)$
for the complex conjugation $c$. 
\end{conjecture}

\begin{lemma}\label{odd}
The property of being symplectically odd is equivalent to 
$\nu(\rho_\lambda(c))=-1$, where $\nu$ is the similitude character in Section 2.
\end{lemma}
\begin{proof}
One implication is clear. So we assume that $\nu(\rho_{F,\lambda}(c))=-1$. 
Let $V$ be the representation space of $\rho_{F,\lambda}$. 
It is easy to see that $\rho_{F,\lambda}(c)$ has eigenvalues $1,1,-1,-1$. (See Lemma 2.3 and Lemma 2.5 of \cite{dieulefait}.) 
Let $v$ be an eigenvector in $V$ for the eigenvalue 1 and $\{e_1,e_2,f_1,f_2\}$ be the symplectic basis with respect to $J$. 
It is easy to see that there exists a matrix $P\in GSp_4(\bF_\lambda)$ such that $Pv=e_1$. 
Then we may assume that 
$$\rho_{F,\lambda}(c)\stackrel{\tiny{GSp_4(\bF_\lambda)}}{\sim}
\begin{pmatrix}
 1& 0 & 0 & 0\\
0 & a & 0& b \\
 0 & 0 & t& 0 \\
 0 & c & 0& d 
\end{pmatrix}
\begin{pmatrix}
 1& x_1 & x_3 & x_2\\
0 & 1 & x_2& 0 \\
 0 & 0 & 1& 0 \\
 0 & 0 & -x_1& 1 
\end{pmatrix}
\in M_Q(\bF_\lambda).
$$
Since $\rho_{F,\lambda}(c)$ is of order 2 and has eigenvalues $1,1,-1,-1$, one has $t=ad-bc=-1$. 
The unipotent part of RHS is preserved by the conjugation of the matrix of the form 
$\begin{pmatrix}
 1& 0 & 0 & 0\\
0 & x & 0& y \\
 0 & 0 & 1& 0 \\
 0 & z & 0& w 
\end{pmatrix}$ with $xw-yz=1$. Hence we have 
$$\rho_{F,\lambda}(c)\stackrel{\tiny{GSp_4(\bF_\lambda)}}{\sim}
\begin{pmatrix}
 1& 0 & 0 & 0\\
0 & 1 & 0& 0 \\
 0 & 0 & -1& 0 \\
 0 & 0 & 0& -1 
\end{pmatrix}
\begin{pmatrix}
 1& x_1 & x_3 & x_2\\
0 & 1 & x_2& 0 \\
 0 & 0 & 1& 0 \\
 0 & 0 & -x_1& 1 
\end{pmatrix}:=A
\in M_Q(\bF_\lambda).
$$
The condition $\nu(\rho_{F,\lambda}(c))^2=I_4$ implies that $x_1=0$. 
Let $P=\begin{pmatrix}
 1& 0 & -\frac{x_3}{2} & -\frac{x_2}{2}\\
0 & 1 & -\frac{x_2}{2}& 0 \\
 0 & 0 & 1& 0 \\
 0 & 0 & 0& 1 
\end{pmatrix}\in M_Q(\bF_\lambda)$. 
Then one has 
$$\rho_{F,\lambda}(c)\stackrel{\tiny{GSp_4(\bF_\lambda)}}{\sim} A
\stackrel{\tiny{GSp_4(\bF_\lambda)}}{\sim}P^{-1}AP=\diag(1,1,-1,-1)$$
$$\stackrel{\tiny{GSp_4(\bF_\lambda)}}{\sim}s^{-1}_2\diag(1,1,-1,-1)s_2=\diag(1,-1,-1,1).
$$
\end{proof}
%\medskip
\begin{remark}\label{endoscopic}
If $\pi_F=\pi_\infty\otimes\otimes_p' \pi_p$ is endoscopic (i.e., its transfer to $GL_4$ is not cuspidal), then by \cite{roberts}, 
$\pi_F$ is associated to a pair $(\pi_1,\pi_2)$ of 
two automorphic cuspidal representations of $GL_2(\A)$ with the same central character $\varepsilon$ 
via theta lifting. 
Since the L-packet of $\pi_\infty$ is a singleton, by Proposition 4.2-(2) of \cite{roberts}, 
$(\pi_{i})_\infty$ should be tempered, but not essentially square integrable. 
Hence one has $(\pi_{i})_\infty={\rm Ind}^{GL_2(\R)}_{B(\R)}(|\cdot|^{s^{(i)}_1}\ve^{(i)}_1, 
|\cdot|^{s^{(i)}_2}\ve^{(i)}_2),\ i=1,2$ where $s^{(i)}_j\in \C$ and $\ve_i$ is 1 or $\sgn$.  
Comparing Langlands parameters, 
one can see that $\pi_i$ has to correspond to an elliptic newform $f_i$ of weight one.  
Thus there exists a finite set $S$ of rational primes which includes all ramified prime of $\pi_F,\ \pi_{f_1}$, and $\pi_{f_2}$ 
so that 
$$L(s,\pi_\infty)=L(s,\pi_{f_1,p})L(s,\pi_{f_2,p}),\: {\rm for\ any}\ p\not\in S.
$$
By Deligne-Serre \cite{d&s}, each $f_i$ gives rise to a unique Artin representation $\rho_{f_i}:G_\Q\lra GL_2(\C)$. 
Hence we may put $\rho_{F}:=\rho_{f_1}\oplus \rho_{f_2}$. 
We define the endoscopic subgroup of $GSp_4$ by 
$$H^{{\rm en}}:=\Bigg\{g=\begin{pmatrix}
 a & 0& b & 0\\
 0  & x& 0 & y \\
 c &0 &  d & 0\\
 0& z & 0 & w
\end{pmatrix}\in GSp_4\Bigg\}\simeq \{(A,B)\in GL_2\times GL_2\ |\ \det A=\det B\},
$$
where the isomorphism is given by $g\mapsto \Bigg(\begin{pmatrix}
 a& b\\
 c& d
\end{pmatrix},\ \begin{pmatrix}
 x& y\\
 z& w
\end{pmatrix}\Bigg)$. 
Since the central characters of $f_i$ are the same, we have $\det(\rho_{f_1})=\det(\rho_{f_2})$. Hence  
the image of $\rho_{F}$ is actually in $H^{{\rm en}}(\C)$. Further it is easy to see that
 $\rho_{F}$ is symplectically odd. 
\end{remark}

Let $\iota:GSp_4\hookrightarrow GL_4$ be the natural embedding. In what follows, we describe the image of semisimplification of 
$\iota\circ \rho_{F,\lambda}:G_\Q\lra GL_4(\bF_\lambda)$. 

\begin{prop}\label{mod-p-galois} Let $\rho_{F,\lambda} : G_\Q\lra GSp_4(\bF_\lambda)$ be as in Conjecture \ref{galois}. Then
one of the following holds:

\begin{enumerate}
\item $\iota\circ\rho_{F,\lambda}$ is absolutely irreducible and Im$\rho_{F,\lambda}$ is contained in $GSp_4(\F_\lambda)$,  
\item $\iota\circ\rho_{F,\lambda}$ is irreducible but not absolutely irreducible and there exists a finite extension $\Bbb F_{\lambda}'/\Bbb F_{\lambda}$, and an  
 absolutely irreducible representation $\sigma: G_\Q\lra GL_{n}(\F'_\lambda)$ with 
\newline $4= n [\F'_\lambda:\F_\lambda], n\not=4$ 
so that $\iota\circ\rho_\lambda=\ds\prod_{\tau\in {\rm Gal}(\F'_\lambda/\F_\lambda)}{}^\tau\sigma$, where 
${}^\tau\sigma(g)=\tau(\sigma(g))$ for $g\in G_\Q$,
\item Im $\rho_{F,\lambda}$ is contained in $M_\ast(\kappa),\ \ast\in \{B,P\}$ where $\kappa$ is a finite extension over $\F_\lambda$ with the 
degree at most $4$. 
\item Im $\rho_{F,\lambda}$ is contained in $M_Q(\kappa')$ or $H^{{\rm en}}(\kappa')$ where $\kappa'$ is a finite extension over $\F_\lambda$ with the degree at most $2$. 
\end{enumerate}
\end{prop}
\begin{proof} 
There exists a finite extension $\F_\lambda'/\F_\lambda$ such that $\rho_{F,\lambda}: G_\Q\lra GSp_4(\F_\lambda')$. 
If $\iota\circ \rho_{F,\lambda}$ is irreducible, then so is $\rho_{F,\lambda}$. Let $\varepsilon=\chi^{-1}_2$.
Then by the symplectic pairing furnished on $\rho_{F,\lambda}$ by Conjecture \ref{galois}, we have an isomorphism 
$$\rho^\vee_{F,\lambda}\simeq  \rho_{F,\lambda}\otimes \varepsilon^{-2}.
$$ 
By Lemma 6.13 of \cite{d&s}, $\iota\circ \rho_{F,\lambda}$ is isomorphic to an irreducible representation 
$\Phi:G_\Q\lra  GL_4(\F_\lambda)$. 
By Chebotarev density theorem, we have an isomorphism between $\Phi^\vee$ and $\Phi\otimes \varepsilon^{-2}$ 
as $\F_\lambda[G_\Q]$-modules. 
We now divide into two cases. 
If $\Phi$ is not absolutely irreducible, this corresponds to the second claim and 
it is easy to prove it. So we assume that $\Phi$ is absolutely irreducible.
Then by Schur's lemma, one has (dropping the action of the character in notation for simplicity)
$$\F_\lambda={\rm End}_{\F_\lambda[G_\Q]}(\Phi)=(\Phi^\vee\otimes\Phi)^{G_\Q}=(\Phi\otimes\Phi)^{G_\Q}=({\rm Sym}^2\Phi)^{G_\Q}\oplus 
(\wedge^2\Phi)^{G_\Q}.$$
Hence 
$$\text{$({\rm Sym}^2\Phi)^{G_\Q}={\rm Bil}^{{\rm sym}}_{\F_\lambda[G_\Q]}(\Phi\times\Phi,\F_\lambda)=\F_\lambda$, or $(\wedge^2\Phi)^{G_\Q}=
{\rm Bil}^{{\rm anti-sym}}_{\F_\lambda[G_\Q]}(\Phi\times\Phi,\F_\lambda)=\F_\lambda$}
$$ 
where 
${\rm Bil}^{{\rm sym}}_{\F_\lambda[G_\Q]}(\Phi\times\Phi,\F_\lambda)$ 
(resp. ${\rm Bil}^{{\rm anti-sym}}_{\F_\lambda[G_\Q]}(\Phi\times\Phi,\F_\lambda)$) is 
the space consisting of all symmetric (resp. anti-symmetric) bilinear forms which commute with the Galois action. 
This means that there exists the symmetric or symplectic structure on $\Phi$. 
On the other hand, there exists a matrix $A\in GL_4(\F_\lambda')$ such that $\Phi=A^{-1}\rho_{F,\lambda} A$. 
Since the conjugate by an element of  $GL_4(\F_\lambda')$ preserves the symmetric or symplectic structure, 
we have ${\rm Bil}^{{\rm anti-sym}}_{\F_\lambda[G_\Q]}(\Phi\times\Phi,\F_\lambda)=\F_\lambda$. 

Next we consider the reducible cases. Let $\{e_1,e_2,f_1,f_2\}$ be the standard symplectic basis corresponding to $J$. 
We assume that $\rho_{F,\lambda}$ has an one dimensional subspace $V_1$ which is stable under the action of $G_\Q$. 
Fix a non-zero $v\in V_1$.  
Then it is easy to see that there exits a matrix $P\in GSp_4(\F_\lambda')$ so that $Pe_1=v$. Hence we may assume that 
$(\rho_{F,\lambda})^{{\rm ss}}=\varepsilon_1\oplus \rho'\oplus \varepsilon_2 \subset M_P(\F_\lambda')$ where $\varepsilon_i: G_\Q\lra {\F_\lambda'}^\times (i=1,2)$ is a character and 
$\rho'$ is a 2-dimensional mod $\ell$ representation of $G_\Q$. Let ${\F_\lambda}''=\F_\lambda(\varepsilon_1,\varepsilon_2 )$ and $\kappa={\F_\lambda}''\cap \F_\lambda'$. 
Then $\kappa$ is of degree at most 4 over $\F_\lambda$. 
Applying Lemma 6.13 of \cite{d&s} to $\rho'$, there exists a matrix $P\in GSp_4(\Bbb F_\lambda)$ so that 
$P^{-1}(\rho_{F,\lambda})^{{\rm ss}}P=\varepsilon_1\oplus \rho''\oplus \varepsilon_2\subset M_P(\kappa)$ where 
$\rho''$ is a 2-dimensional mod $\ell$ representation of $G_\Q$ over $\kappa$. 
It is the same in the case that $\rho_{F,\lambda}$ has a three dimensional subspace $V_3$ which is stable under the action of $G_\Q$ 
by taking the duality with respect to the symplectic pairing on $\rho_{F,\lambda}$ into account. 

Finally we consider the case that $\rho_{F,\lambda}$ has an 2-dimensional irreducible subspace $V_2$ which is stable under the action of $G_\Q$. 
Let $r$ be the dimension of the kernel of the linear map $V_2\lra V^\ast_2, v\mapsto \langle \ast,v \rangle$. 
It is easy to see that $r=1$ or 2. 

First we assume $r=2$. Fix a basis $\{v_1,v_2\}$ of $V_2$. One can easily find vectors $w_1,w_2\in V$ so that 
$\langle v_1,w_1\rangle=\langle v_2,w_2\rangle=1$ and $\langle w_1,w_2\rangle=0$ since $V_2\stackrel{\sim}{\lra} V^\ast_2$. 
Then we may assume that $V_2=\langle e_1,e_2\rangle$ or $V_2=\langle f_1,f_2\rangle$. 
In this case, we may have $(\rho_{F,\lambda})^{{\rm ss}}\subset M_P(\kappa')$ giving the claim by Lemma 6.13 of \cite{d&s} again. 
Here $\kappa'$ is a finite extension of $\F_\lambda$ with the degree at most 2.   

Next we assume $r=1$. Take a non-zero vector $v$ in the kernel of the map and denote by $v^\ast$ the dual basis vector of $v$ which is 
identified as a vector in $V$ by the pairing. 
Then $\langle v,v^\ast\rangle=1$. Take $w\in V_2$ (and denote by $w^\ast$ the dual vector of $w$) so that $\langle v,w\rangle=0$. 
This gives us $\langle v^\ast,w^\ast\rangle=0$. Hence $\{v,w,v^\ast,w^\ast\}$ makes the standard symplectic basis.  
Therefore one may have that $V_2=\langle e_1,f_1\rangle$ or $V_2=\langle e_2,f_2\rangle$. 
In this case, one has $(\rho_{F,\lambda})^{{\rm ss}}\subset H^{{\rm en}}(\kappa')$ giving the claim 
by Lemma 6.13 of \cite{d&s} again. 
\end{proof}

\section{Bounds of certain subgroups of $GSp_4(\F_{\ell^n})$}\label{subgroups1}
In this section, we will study the bounds of certain subgroups of $GSp_4(\F_{\ell^n})$ for odd prime $\ell$ and $n\ge 1$. 
For a finite set $X$, we denote by $|X|$, the cardinality of $X$.  

By imitating the strategy of \cite{d&s} for $GL_2(\F_\ell)$, we consider the following property of a subgroup $G$ of $GSp_4(\F_{\ell^n})$.

\begin{defin}\label{finite-subgroup}
Let $M$ and $\eta$ $(0<\eta<1)$ be positive constants.
$$
C(\eta,M):\ \mbox{there exists a subset $H$ of $G$ such that}\ 
\left\{\begin{array}{ll}
(i)\  |H|\ge (1-\eta)|G|, \\
(ii)\  |\{\det(1-hT)\in \F_\ell[T]|\ h\in H \}|\le M.
\end{array}\right. 
$$
\end{defin}
Then the following lemma is easy to prove.
\begin{lemma}\label{easy-lemma}(cf. the proof of Proposition 7.2 in \cite{d&s}) Let $G$ be a finite group with a subgroup $G'$ of index 2. 
 Then if $G$ satisfies $C(\eta,M)$, then $G'$ satisfies $C(2\eta,M)$. 
\end{lemma}
\begin{proof} Let $H$ be a subset of $G$ which satisfies the property $C(\eta,M)$. Let $H'=H\cap G'$. 
Then $|H'|\geq (1-\eta)|G|=(2-2\eta)|G'|\geq (1-2\eta)|G'|$. The second condition is obvious.
\end{proof}

We denote by $M_\ast$, the Levi factor of the parabolic subgroup $\ast\in \{B,P,Q\}$. 
Recall $M_B(\F_{\ell^n})$, $M_P(\F_{\ell^n})$, and $M_Q(\F_{\ell^n})$ from Section 2. 
Recall also $H^{{\rm en}}(\F_{\ell^n})$ from Section 7.

For a subgroup $G$ of $GSp_4(\F_\ell)$, we say $G$ is semisimple if the identity representation $G\hookrightarrow GSp_4(\F_{\ell^n}) \hookrightarrow  GL_4(\F_{\ell^n})$ is semisimple.

We need the classification of all semisimple subgroups of $GSp_4(\F_{\ell^n}),\ n\ge 1$. All of them are 
the semisimple parts of groups taken from \cite{dieulefait} and \cite{dieulefait1} though 
some of explicit forms are not given there. 

\begin{lemma}\label{semisimple-subgroup}
Let $G$ be a semisimple subgroup of $GSp_4(\F_{\ell^n}),\ n\ge 1$. Then up to conjugacy, $G$ is one of the following:

\hspace{0.45in} (reducible cases)

\begin{enumerate}
\item $G$ is contained in $M_\ast(\F_{\ell^n})$ for some $\ast\in \{B,P,Q\}$. 
\item $G$ is contained in $H^{{\rm en}}(\F_{\ell^n})$. 

\smallskip 
(irreducible cases and $n=1$)

\item $G$ contains $Sp_4(\F_\ell)$, 
\item $G$ is contained in ${\rm Sym}^3GL_2(\F_\ell)$, 
\item $G$ is contained in  $H_\ell:=\langle M_P(\F_\ell), \begin{pmatrix}
 0& I_2\\
 I_2& 0
\end{pmatrix} \rangle$, but $G\not\subset M_P(\F_\ell)$. 

\item $G$ is contained in  $H_\ell:=\langle H^{{\rm en}}(\F_\ell), \begin{pmatrix}
 0& I_2\\
 I_2& 0
\end{pmatrix} \rangle$, but $G\not\subset H^{{\rm en}}(\F_\ell)$. 

\item Fix a quadratic non-residue $u\in \F_\ell$ and a square root $\sqrt{u}\in\F_{\ell^2}$ of $u$.  Choose a solution $(a,b)\in \F^\times_\ell\times \F^\times_\ell$ so that $a^2+b^2=u$. Then 
for $a_i=x_i+y_i\sqrt{u}\in \F^\times_{\ell^2},\ x_i,y_i\in \F_\ell\ (i=1,\ldots,4)$, let $S(a_i)=\begin{pmatrix}
 x_i+ay_i& b y_i \\
 by_i& x_i-a y_i
\end{pmatrix}$. Note that ${}^tS(a_i)=S(a_i)$ and 
$\begin{pmatrix}
 S(a_1)& S(a_2)\\
 S(a_3)& S(a_4)
\end{pmatrix}\in GSp_4(\F_\ell)$ if and only if $a_1 a_4-a_2 a_3\in \F^\times_\ell$. Then $G$ is contained in 
$$\Bigg\langle \Bigg\{\begin{pmatrix}
 S(a_1)& S(a_2)\\
 S(a_3)& S(a_4)
\end{pmatrix}\in GSp_4(\F_\ell) \ \Bigg|\ a_i\in \F^\times_{\ell^2} \Bigg\}, \begin{pmatrix}
 0& I_2\\
 I_2& 0
\end{pmatrix} \Bigg\rangle\simeq S_\ell\rtimes \{\pm 1\},
$$
where $S_\ell:=\left\{ \begin{pmatrix} a_1 &a_2\\a_3&a_4
\end{pmatrix}\right\}=\{g\in GL_2(\F_{\ell^2})\ |\ \det(g)\in \F^\times_\ell\}$.

\item Fix a quadratic non-residue $u\in \F_\ell$ and choose a solution $\lambda\in \F^\times_{\ell^2}$ so that $\lambda^2=u$. 
Then $G$ is contained in 
$$\Bigg\langle \Bigg\{ u(A,B)=\begin{pmatrix}
 A& B\\
 u B& A
\end{pmatrix}\in GSp_4(\F_\ell)   \Bigg\},  \left(\begin{array}{cc}
 0& I_2\\
 I_2& 0
\end{array}
\right)\Bigg\rangle\simeq GU_2(\F_{\ell})\rtimes \{\pm 1\}$$
where $ GU_2(\F_{\ell})=\{g\in GL_2(\F_{\ell^2})\ |\ \sigma( {}^t g)g=\nu I_2,\ \nu\in \F^\times_\ell  \}$ and $\sigma$ is 
the generator of ${\rm Gal}(\F_{\ell^2}/\F_\ell)$. 
The matrices $A,B\in M_2(\F_\ell)$ satisfy $A {}^tA-u B {}^tB=\nu I_2,\ \nu \in\F^\times_\ell$ and $A {}^tB-B {}^tA=0$. 
Then the above isomorphism is given by $u(A,B)\mapsto A+\lambda B$,  

\item $G$ is contained in 
$\Bigg\{\begin{pmatrix}
 av & 0& bv & 0\\
 0  & az& 0 & bz \\
 cv &0 &   dv & 0\\
 0& cz & 0 & dz
\end{pmatrix}\in GSp_4(\F_\ell)\Bigg\}\bigcup 
\Bigg\{\begin{pmatrix}
 0   & av & 0& bv  \\
 az & 0 & bz & 0 \\
 0  & cv &0 &   dv \\
 cz & 0 & dz & 0
\end{pmatrix}\in GSp_4(\F_\ell)\Bigg\}$,
which is realized by taking the tensor product of $GL_2(\F_\ell)$ and a dihedral subgroup 
$D= \left(\begin{array}{cc}
 \ast& 0\\
 0& \ast 
\end{array}
\right)\bigcup \left(\begin{array}{cc}
 0& \ast\\
 \ast& 0 
\end{array}
\right)$ of $GL_2(\F_\ell)$. 
  
\item We denote by $\overline{G}$, the image of $G$ in $PGSp_4(\F_\ell)$. Then $\overline{G}$ is isomorphic to $A_6,\ S_6$, or $A_7$,
 or there exists a normal abelian subgroup $E$ of $\overline{G}$ with order $16$ 
so that $\overline{G}/E\simeq A_5$ or $S_5$. 
\end{enumerate}
\end{lemma} 

We prove the following key proposition by a case by case analysis with the help of the above Lemma.
\begin{prop}\label{main-prop-finite-group} For positive constants $M$ and $\eta,\ (0<\eta<\frac{1}{2})$, there exists a constant $A=A(\eta,M)$ such that 
for every rational odd prime $\ell$ and every semisimple subgroup $G$ of 
$M_\ast(\F_{\ell^4}),\ \ast\in\{B,P\}$, $M_Q(\F_{\ell^2}),\ H^{{\rm en}}(\F_{\ell^2})$, or $GSp_4(\F_\ell)$, or
$\ds\prod_{\tau\in {\rm Gal}(\F_{\ell^m}/\F_\ell)} {}^\tau(GL_n(\F_{\ell^m}))\subset GL_4(\F_\ell)$ with $nm=4, n\not=4$, 
satisfying $C(\eta,M)$, we have $|G|<A$. 
\end{prop}

\begin{proof}
Case (1)-$M_B(\F_{\ell^4})$: At most 8 ($=|W_G|$ where $W_G$ is the Weyl group of $G$) elements of $M_B(\F_\ell)$ have a given characteristic polynomial. 
The hypothesis $C(\eta,M)$ (with $0<\eta<1$) gives 
$$(1-\eta)|G|\le |H|\le 8|\{\det(1-hT)\in \F_\ell[T]\ |\ h\in H\}|\le 8M.
$$
giving a bound $|G|<\frac{8M}{1-\eta}$.
\medskip

Case (1)-$M_P(\F_{\ell^4})$: 
In this case,  the conjugacy classes are isomorphic to the product of the conjugacy classes of $GL_2(\F_{\ell^n})$ and $\F^\times_{\ell^n}$. Hence the similitude 
does not essentially affect the result. 
It is easy to generalize Proposition 7.2 of \cite{d&s} to the case $GL_2(\F_{\ell^n})$ for any $n\ge 1$. 
Let $A$ be the analogous constant of Proposition 7.2 of \cite{d&s} in the case $GL_2(\F_{\ell^4})$. 
Then we have $|G|\le 2A$ by taking the action of the element 
$ \left(\begin{array}{cc}
 0& I_2\\
 I_2& 0 
\end{array}
\right)$ into account.  

Case (1)-$M_Q(\F_{\ell^2})$ is the same as well. 

\medskip
 
Case (2): Let ${\rm pr}_i:H^{{\rm en}}(\F_{\ell^2})\simeq \{(A,B)\in GL_2(\F_{\ell^2})\times GL_2(\F_{\ell^2})|\ \det A=\det B\}\stackrel{{\rm pr}_i}{\lra} GL_2(\F_{\ell^2}) $ be  
 the $i$-th projection for $i=1,2$. Note that there is an exact sequence 
\begin{equation}\label{go}
1\lra SL_2(\F_{\ell^2})\times SL_2(\F_{\ell^2})\lra H^{{\rm en}}(\F_{\ell^2})\lra \F^\times_{\ell^2}\lra 1
\end{equation}
 by an obvious way. 
Then  ${\rm pr}_i(G)$ satisfies one of the conditions (a), (b), (c), or (d) of Proposition 7.2 of \cite{d&s}. So essentially there are at most 6 possibilities of $G$. 
For simplicity, we say ($\ast_1$)-($\ast_2$) case for $\ast_1,\ast_2\in\{a,b,c,d\}$  if ${\rm pr}_1(G)$ satisfies ($\ast_1$) and ${\rm pr}_2(G)$ satisfies $(\ast_2)$. 
We recall the following fact which is easy to prove:
\begin{equation}\label{fact}
\mbox{There are at most $\ell^{2n}+\ell^n$ elements of $GL_2(\F_{\ell^n})$ which have a given characteristic polynomial.}
\end{equation}
(see the proof of Proposition 7.2 of \cite{d&s} for $n=1$).

%\medskip
(a)-(a) case:  Let $r:=[G:SL_2(\F_{\ell^2})\times SL_2(\F_{\ell^2})]$. Then $|G|=r\ell^4(\ell^4-1)^2$. 
It is easy to see that the characteristic polynomial of any element $g$ of $G$ which corresponds to $(A,B)\in GL_2(\F_{\ell^2})\times 
GL_2(\F_{\ell^2}),\ \det A=\det B$, is of the form 
$$\Phi_g(T)=\Phi_A(T)\Phi_B(T).$$ 
By (\ref{fact}),  at most $8(\ell^4+\ell^2)^2$ elements of $G$ have a given characteristic polynomial. 
If $G$ satisfies $C(\eta,M)$, one has  
$$(1-\eta)r\ell^4(\ell^4-1)^2\le 8(\ell^4+\ell^2)^2M,$$
giving $$(1-\eta)r\ell\le (1-\eta)r(\ell^2-1)^2\le 8M\ {\rm and}\ \ell\le \frac{8M}{1-\eta}.$$
Hence we have the bound of $|G|$ which depends only on $M$ and $\eta$. 
%\medskip

(a)-(b) case: There exists a subgroup $K$ of $\F^\times_{\ell^2}$ such that $$G=\Bigg\{g=\Bigg(A, \begin{pmatrix}
 a& 0\\
 0& a^{-1}\det A 
\end{pmatrix} \bigg)\Bigg|\ A\in {\rm pr}_1(G),\ a\in K \Bigg\}.
$$ 
Let $r=[ {\rm pr}_1(G):SL_2(\F_{\ell^2})]$. Then $|G|=|K|r\ell^2(\ell^4-1)$. 
The characteristic polynomial of any element $g$ of $G$ is of the form 
$$\Phi_g(T)=\Phi_A(T)(T-a)(T-a^{-1}\det A).
$$
Then by (\ref{fact}), at most $2(\ell^4+\ell^2)$ elements of $G$ have a given characteristic polynomial. 
If $G$ satisfies $C(\eta,M)$, one has  
$$(1-\eta)|K|r\ell^2(\ell^4-1)\le 2(\ell^4+\ell^2)M,$$
giving $$(1-\eta)r|K|\ell\le (1-\eta)r|K|(\ell^2-1)\le 2M\ {\rm and}\ \ell\le \frac{2M}{1-\eta}.$$
Hence we have the bound of $|G|$ which depends only on $M$ and $\eta$. 

%\medskip

(a)-(c) case: This case is reduced to the case (a)-(b) by Lemma \ref{easy-lemma}. 

%\medskip
(a)-(d) case: In this case, the group  $K={\rm pr_2}(G)\cap SL_2(\F_{\ell^2})$ is of order at most 120, whence has at most 120 elements of 
the given determinant. 
Let $r=[ {\rm pr}_1(G):SL_2(\F_{\ell^2})]$. Then $|G|=|K|r\ell^2(\ell^4-1)$ by (\ref{go}).  
Then by (\ref{fact}), at most $120(\ell^4+\ell^2)$ elements of $G$ have a given characteristic polynomial. 
If $G$ satisfies $C(\eta,M)$, one has  
$$(1-\eta)|K|r\ell^2(\ell^4-1)\le 120(\ell^4+\ell^2)M,$$
giving $$(1-\eta)r|K|\ell\le (1-\eta)r|K|(\ell^2-1)\le 120M\ {\rm and}\ \ell\le \frac{120M}{1-\eta}.$$
Hence we have the bound of $|G|$ which depends only on $M$ and $\eta$. 

%\medskip
(b)-(b) case: Any element of $G$ is of the form $\Bigg( \begin{pmatrix}
 a& 0\\
 0& a^{-1}c 
\end{pmatrix},\, \begin{pmatrix}
 b& 0\\
 0& b^{-1}c  
\end{pmatrix} \Bigg)$. Hence at most $8$ elements of $G$ have a given characteristic polynomial. 
Then one has $|G|\le \frac{8M}{1-\eta}.$

%\medskip
(b)-(c) case: This case is reduced to the case (b)-(b) by Lemma \ref{easy-lemma}. 

%\medskip
(b)-(d) case:  By the analysis of (a)-(d) case, we see that at most $120\times 2=240$ elements of $G$ have a given characteristic polynomial. 
Hence we have $|G|\le \frac{240M}{1-\eta}.$

%\medskip
(c)-(d) case: This case is reduced to the case (b)-(d) by Lemma \ref{easy-lemma}. 

\medskip
For the case $\ds\prod_{\tau\in {\rm Gal}(\F_{\ell^m}/\F_\ell)} {}^\tau(GL_n(\F_{\ell^m}))\subset GL_4(\F_\ell)$ with $nm=4, n\not=4$,
it is reduced to 
the case (1)-$M_P(\F_{\ell^4})$. So we omit the proof. 

\medskip
Case (3): Let $r:=[G:Sp_4(\F_\ell)]$. Then $|G|=r\ell^4(\ell^4-1)(\ell^2-1)$. By Table 1 and Table 2 of \cite{shinoda}, one can compute 
 the number of elements of $G$ which have a given characteristic polynomial. As a result, such number is at most $C\ell^8$ for some positive constant $C$ 
which is independent of $\ell$.  
For instance, if the semi-simple part of $g\in G$ is ${\rm diag}(a,a,a,a),\ a\in \F^\times_\ell$, from the centralizer of 
the elements of types $A_0,\ A_1,\ A_{21}, A_{22}$, and $A_{3}$ of Table 2 in \cite{shinoda}, the number of elements of $G$ with the 
characteristic polynomial $(T-a)^4$ is computed as the sum of 
orbits of each types:
$$\frac{|GSp_4(\F_\ell)|}{|GSp_4(\F_\ell)|}+\frac{|GSp_4(\F_\ell)|}{\ell^4(\ell-1)(\ell^2-1)}+\frac{|GSp_4(\F_\ell)|}{2\ell^3(\ell-1)^2}+\frac{|GSp_4(\F_\ell)|}{2\ell^3(\ell^2-1)}
+\frac{|GSp_4(\F_\ell)|}{\ell^2(\ell-1)}=\ell^8-\frac{1}{2}\ell^6+\ell^5-\frac{1}{2}\ell^4+\frac{1}{2}\ell^2-\ell+\frac{1}{2}.$$
If $G$ satisfies $C(\eta,M)$, one has  
$$(1-\eta)r\ell^4(\ell^4-1)(\ell^2-1)\le C\ell^8M,$$
giving 
$$(1-\eta)r\ell\le (1-\eta)r(\ell^2-1)\le CM\ {\rm and}\ \ell\le \frac{CM}{1-\eta}.
$$
Hence we have the bound of $|G|$ which depends only on $M$ and $\eta$. 
%\medskip
 
Case (4): It is reduced to the case $GL_2(\F_\ell)$. 
%\medskip
 
Cases (5) and (6): These cases are reduced to the cases (1)-$M_P$ and (2) for $n=1$ by Lemma \ref{easy-lemma}. 
%\medskip
 
Case (7): Let
 $S(A):=\begin{pmatrix}
 S(a_1)& S(a_2)\\
 S(a_3)& S(a_4)
\end{pmatrix}$ 
for 
$A=\begin{pmatrix}
 a_1 & a_2\\
 a_3 & a_4
\end{pmatrix}\in S_\ell$. Then it is easy to see that $\Phi_{S(A)}(T)=\Phi_A(T)\sigma(\Phi_A(T))$ where 
$\Phi_\ast$ means the characteristic polynomial of $\ast$ and $\sigma$ is 
the generator of ${\rm Gal}(\F_{\ell^2}/\F_\ell)$. As in the proof of the reducible case (2) (replacing the base field by $\F_{\ell^2}$), we have three possibilities for $G$. 
We give a proof for the case when $G\cap S_{\ell}$ contains $SL_2(\F_{\ell^2})$. 
The other cases are similar. 

Let $r=[G:SL_2(\F_{\ell^2})]$ so that $|G|=r\ell^2(\ell^4-1)$.  
Then by (\ref{fact}), at most $4(\ell^4+\ell^2)$ elements of $G$ have a given characteristic polynomial. 
Here the factor $4$ comes from the orders of $\{\pm 1\}$ and ${\rm Gal}(\F_{\ell^2}/\F_\ell)$. 
If $G$ satisfies $C(\eta,M)$, one has  
$$(1-\eta)r\ell^2(\ell^4-1)\le 4(\ell^4+\ell^2)M,$$
giving $$(1-\eta)r\ell\le (1-\eta)r(\ell^2-1)\le 8M\ {\rm and}\ \ell\le \frac{8M}{1-\eta}.$$
Hence we have the bound of $G$ which depends only on $M$ and $\eta$.   
 
 %\medskip
Case (8): Let $U:=A+\lambda B$. 
Then it is easy to see that $\Phi_{g(A,B)}(T)=\Phi_U(T)\sigma(\Phi_U(T))=\Phi_U(T)\Phi_{\sigma(U)}(T)$. 
Since $G$ is irreducible, 
the composition $G\cap GU_2(\F_{\ell})\hookrightarrow GL_2(\F_{\ell^2})$ is also irreducible and it has 
 three possibilities as in case (2). We give a proof for the case when $G\cap GU_2(\F_{\ell})$ contains $SL_2(\F_{\ell^2})\cap GU_2(\F_{\ell})=SU_2(\F_\ell)$. 
The other cases are similar.  Note that $SU_2(\F_\ell)\simeq SL_2(\F_\ell)$ 
(the isomorphism is considered in $GL_2(\bF_\ell)$.), and hence 
$|SL_2(\F_{\ell^2})\cap GU_2(\F_{\ell})|=\ell(\ell^2-1)$. 
As in the case $GL_2(\F_\ell)$, it is not so hard to show that the number of elements in $GU_2(\F_\ell)$ with the 
given polynomial is $\ell^2+\ell,\ \ell^2$, or $\ell^2+\ell$ as the polynomial in question has 2, 1, or 0 roots in $\Bbb F_\ell$, resp. 
Let $r=[G:SU_2(\F_\ell)]$ so that $|G|=r\ell(\ell^2-1)$.  
Then at most $4(\ell^4+\ell^2)$ elements of $G$ have a given characteristic polynomial. 
Here the factor $4$ comes from the orders of $\{\pm 1\}$ and ${\rm Gal}(\F_{\ell^2}/\F_\ell)$. 
If $G$ satisfies $C(\eta,M)$, one has  
$$(1-\eta)r\ell(\ell^2-1)\le 4(\ell^2+\ell)M,$$
giving $$(1-\eta)r\ell\le 8M\ {\rm and}\ \ell\le \frac{8M}{1-\eta}.$$
Hence we have the bound of $G$ which depends only on $M$ and $\eta$.   

 %\medskip
Case (9): Since $G$ is contained in the tensor representation $GL_2(\F_\ell)\otimes D$,  
where $D$ is a dihedral subgroup of $GL_2(\F_\ell)$, it is reduced to the case $GL_2(\F_\ell)$ by Lemma \ref{easy-lemma}. So we omit the proof.  
 
 %\medskip
Case (10): Among the finite groups appearing in case (10), $A_7$ is the largest: $|A_7|=2520$. 
 The group $G\cap SL_4(\F_\ell)$ is of order at most $4\times 2520$, whence $G$ has at most $10080$ elements with the 
given characteristic polynomial. 
If $G$ satisfies $C(\eta,M)$, one has  
$$(1-\eta)|G|\le 10080M,$$
giving the bound of $G$. This completes the proof.   
\end{proof}

\section{Proof of the Main Theorem}
In this section we give a proof of the main theorem (Theorem \ref{artin-rep}). 
Let $\pi_F=\pi_{\infty}\otimes \otimes_p' \pi_p$ be the cuspidal automorphic representation of $GSp_4(\A)$ attached to the real analytic Siegel cusp form of weight (2,1).
By Lemma \ref{CAP}, such $\pi_F$ is not a CAP representation. 
Let $\Q_F$ be the Hecke field of $F$, and let $L$ be the Galois closure of $\Q_F$. By the assumption (Rat), 
$L$ is a finite extension of $\Q$.  
We denote by $S_\pi$ the set of rational primes consisting of primes $p$ so that $\pi_p$ is ramified.  
Let $P_L$ be the set prime numbers $\ell$ which splits completely in $L$. 
For each $\ell\in P_L$, choose a finite place $\lambda_\ell$ of $L$ dividing $\ell$. 
By Conjecture \ref{galois}, there exists a continuous semi-simple representation 
$$\rho_\ell:=\rho_{\lambda_\ell}:G_\Q\lra GSp_4(\overline{\F}_{\ell})$$
which is unramified outside $S_\pi\cup\{\ell\}$, and 
$$\det(I_4-\rho_\ell({\rm Frob}_p)T)\equiv H_p(T) \, {\rm mod}\ \lambda_\ell,
$$
where $H_p(T)=1-a_{1,p}T+ (pa_{2,p}+(1+p^{-2})\varepsilon(p^{-1}) T^2-\varepsilon(p^{-1})a_{1,p}T^3+\varepsilon(p^{-1})^2T^4$.

Let $G_\ell:={\rm Im}\, \rho_\ell$. 
\begin{lemma}\label{condition}
For any $\eta,\, 0<\eta<1$, there exists a constant $M$ such that 
$G_\ell$ satisfies $C(\eta,M)$ for every $\ell\in P_L$.
\end{lemma}
\begin{proof}
By Proposition \ref{density}, if we let $\mathcal{M}:=\{ H_p(T)\, |\ p\not\in X_\eta\},$
then $\mathcal{M}$ is a finite set. Let $M:=|\mathcal{M}|$ which will be a desired constant. 
Let us consider the subset of  $G_\ell$ defined by 
$$H_\ell:=\{g\in G_\ell\ |\ g \stackrel{\tiny{G_\ell}}{\sim} \rho_\ell({\rm Frob}_p) {\rm\ for\ some\ }p\not\in X_\eta  \}.$$
By Chebotarev density theorem, one has 
$$1=\frac{|H_\ell|}{|G_\ell|}+{\rm den}(X_\eta)\le \frac{|H_\ell|}{|G_\ell|}+{\rm den.sup}(X_\eta)
\le \frac{|H_\ell|}{|G_\ell|}+\eta,$$
giving $(1-\eta)|G_\ell|\le |H_\ell|$. 

The characteristic polynomial of each element of $H_\ell$ is the reduction of some element of $\mathcal{M}$. Therefore one has 
$$|\{\det(I_4-hT)\ |\ h\in H_\ell  \}|\le M.
$$
\end{proof}

By Lemma \ref{condition} together with Proposition \ref{main-prop-finite-group},  
there exists a constant $A$ such that $|G_\ell|\le A$ for any $\ell\in P_L$. 
Let $Y$ be the set of polynomials $(1-\alpha T)(1-\beta T)(1-\gamma T)(1-\delta T)$, where $\alpha,\beta,\gamma,$ and $\delta$ are 
roots of unity of order less than $A$. 
If $p\not\in S_\pi$, for all $\ell\in P_L$ with $\ell\neq p$, there exists $R(T)\in Y$ such that 
$$H_p(T)\equiv R(T) \ {\rm mod}\ \lambda_\ell.
$$
Since $Y$ is finite and $P_L$ is infinite, 
$$H_p(T)=R(T).
$$
Let $P_L'$ be the set of $\ell\in P_L$ such that $\ell>A$ and for $R,S\in Y$, $R\not\equiv S\ {\rm mod}\ \lambda_\ell$. 
Then it is easy to see that $P_L'$ is infinite. 
For each $\ell\in P_L'$, $\ell$ does not divide $|G_\ell|$, since $\ell>A\ge |G_\ell|$.  
Let $\pi: GSp_4(\mathcal{O}_{\lambda_\ell}){\lra}  GSp_4(\F_\ell))$ be the reduction map. 
Applying a profinite version of Schur-Zassenhaus' theorem (cf. \cite{RZ}, page 40, Theorem 2.3.15) to $\pi^{-1}(G_\ell)$ 
and $\pi^{-1}(G_\ell)\cap {\rm Ker}(\pi)$ (note that the latter group is a Hall subgroup of $\pi^{-1}(G_\ell)$ 
in the sense of \cite{RZ}),   
there exists a subgroup $H\subset \pi^{-1}(G_\ell)$ such that $\pi^{-1}(G_\ell)=H\cdot(\pi^{-1}(G_\ell)\cap {\rm Ker}(\pi))$ 
and $H\cap(\pi^{-1}(G_\ell)\cap {\rm Ker}(\pi))={1}$. 
Then the composition of the inclusion $H\hookrightarrow \pi^{-1}(G_\ell)$ and 
$\pi$ induces an isomorphism  
$$H \stackrel{\sim}{\lra} G_\ell={\rm Im}\: \rho_\ell.
$$ 
Hence we have a lift $\rho'_\ell:G_\Q\lra GSp_4(\mathcal{O}_{\lambda_\ell})$ of $\rho_\ell$.  
Since  any element of ${\rm Im}(\rho'_\ell)\simeq H$ is of order less than $A$, one has 
$\det(I_4-\rho'_\ell({\rm Frob}_p)T)\in Y$  for any $p\nmid N\ell$ by construction.  
On the other hand, we have 
$$\det(I_4-\rho'_\ell({\rm Frob}_p)T)\equiv H_p(T)\ {\rm mod}\ \lambda_\ell.
$$
Since $\ell\in P_L'$, the above congruence relation implies the equality 
$$\det(I_4-\rho'_\ell({\rm Frob}_p)T)=H_p(T).
$$  
for any $p\nmid N\ell$. Now we replace $\ell$ with another prime $\ell'\in P_L'$. Then one has 
$\rho'_{\ell'}:G_\Q\lra GSp_4(\mathcal{O}_{\lambda_{\ell'}})$ such that 
$$\det(I_4-\rho'_\ell({\rm Frob}_p)T)=\det(I_4-\rho'_{\ell'}({\rm Frob}_p)T)
$$
for any $p\nmid N\ell\ell'$. 
By Chebotarev density theorem, one has $\iota\circ \rho'_{\ell'}\sim \iota\circ \rho'_{\ell}$ and this means that 
$\rho'_{\ell}$ is unramified at $\ell$. Hence we have the desired representation 
\begin{equation}\label{artin-rep-eq}
\rho_F:=\rho'_{\ell}:G_\Q\lra GSp_4(\mathcal{O}_{\lambda_\ell})\hookrightarrow GSp_4(\C),
\end{equation}
where the second map comes from a fixed embedding $\mathcal{O}_{\lambda_\ell}\hookrightarrow \C$. 
Since $\nu(\rho_F(c))\equiv\ -1\ {\rm mod}\ \ell$ and $\nu(\rho_F(c))$ has eigenvalues $1,1,-1$, and $-1 \ {\rm mod}\ \ell$, by Conjecture \ref{galois}, for all but finitely many $\ell$, one has $\nu(\rho_F(c))=-1$ and $\nu(\rho_F(c))$ has eigenvalues $1,1,-1$, and $-1$. 
This implies $\rho_F$ is symplectically odd by Lemma \ref{odd}.

It remains to show that $\rho_F$ is reducible if and only if $F$ is of endoscopic type.
If $\rho_F$ is reducible, then we have the following four cases: 
\begin{enumerate}
\item Im$\rho_F$ is contained in $M_B(\C)$;
\item Im$\rho_F$ is contained in $M_Q(\C)$, but not in $M_B(\C)$; 
\item Im$\rho_F$ is contained in $M_P(\C)$, but not in $M_B(\C)$; 
\item Im$\rho_F$ is contained in $H^{{\rm en}}(\C)$, but not in $M_B(\C)$. 
\end{enumerate}

We will prove that only the case (4) occurs and further it is the case only when $F$ is of endoscopic type.

Case (1):  One can see that $\rho_F={\rm diag}(\chi_1,\chi_2,\chi^{-1}_1\varepsilon,\chi^{-1}_2\varepsilon)$ where 
$\chi_1,\chi_2:G_\Q\lra \C^\times, i=1,2$ are gr\"ossencharacters of finite order and
$$\lambda(p)=\chi_1(p)+\chi_2(p)+\chi^{-1}_1(p)\varepsilon(p)+\chi^{-1}_2(p)\varepsilon(p),
$$   
for any $p\nmid N$. 
Then we have 
$$|\lambda(p)|^2=4+2(\chi_1\overline{\chi}_2(p)+\overline{\chi}_1\chi_2(p))+2(\overline{\chi}_1\overline{\chi}_2\varepsilon(p)+\chi_1\chi_2\overline{\varepsilon}(p) )
+\overline{\chi}^2_1\varepsilon(p)+\overline{\chi}^2_2\varepsilon(p)+\chi^2_1\overline{\varepsilon}(p)+\chi^2_2\overline{\varepsilon}(p).$$
One then has 
$$\lim_{s\to 1^+}\frac{1}{\log \frac{1}{s-1}}\sum_{p\nmid N}\frac{|\lambda(p)|^2}{p^s}\ge 4$$
which contradicts to (\ref{rankin}). Hence this case does not occur. 

Case (2): One can see that $\rho_F=\chi_1\oplus \rho\oplus \chi_2$ where 
$\chi_1,\chi_2:G_\Q\lra \C^\times, i=1,2$ are gr\"ossencharacters of finite order and $\rho:G_\Q\lra GL_2(\C)$ is an odd irreducible Artin representation. 
By Corollary 0.4 of \cite{kisin}, $\rho$ is modular, i.e., there exists an elliptic cusp form $f$ attached to $\rho$.  
Let $\widetilde{\rho}$ be the complex conjugate of $\rho$, i.e., the composite of $\rho$ and the complex conjugate 
$GL_2(\C)\lra GL_2(\C)$. 
Since  
$$\lambda(p)=\chi_1(p)+\chi_2(p)+{\rm tr}(\rho({\rm Frob}_p))$$   
for any $p\nmid N$, one has  
\begin{eqnarray*}
|\lambda(p)|^2 &=& 2+{\rm tr}(\rho\otimes\widetilde{\rho}({\rm Frob}_p)) \\
& + & 2(\chi_1\overline{\chi}_2(p)+\overline{\chi}_1\chi_2(p))+(\chi_1(p)+\chi_2(p)){\rm tr}(\widetilde{\rho}({\rm Frob}_p))+
(\overline{\chi}_1(p)+\overline{\chi}_2(p)){\rm tr}(\rho({\rm Frob}_p)).
\end{eqnarray*}
A standard argument on Rankin-Selberg convolution of $f$ shows that 
 $$\lim_{s\to 1^+}\frac{1}{\log \frac{1}{s-1}}\sum_{p\nmid N}\frac{{\rm tr}(\rho\otimes\widetilde{\rho}({\rm Frob}_p))}{p^s}=1.
$$
Then one has $$\lim_{s\to 1^+}\frac{1}{\log \frac{1}{s-1}}\sum_{p\nmid N}\frac{|\lambda(p)|^2}{p^s}\ge 3$$
which contradicts to (\ref{rankin}). Hence this case does not occur. 

Case (3): One can see that $\rho_F=\rho^{\vee}\otimes\rho\otimes \chi=\chi\oplus Ad(\rho)\otimes\chi$ where 
$\chi:G_\Q\lra \C^\times$ is a gr\"ossencharacter of finite order and $\rho: G_\Q\lra GL_2(\C)$ is an odd irreducible Artin representation. Let $f$ be the elliptic cusp form attached to $\rho$ explained as above, and let $Ad(\pi_f)$ be the Gelbart-Jacquet lift of $\pi_f$. Then the transfer of $\pi_F$ to $GL_4$ is of the form $\chi\boxplus Ad(\pi_f)\otimes\chi$. This contradicts to the fact that $\Pi$ is either cuspidal or $\pi_1\boxplus\pi_2$.

Case (4): One can see that $\rho_F=\rho_1\oplus \rho_2$ where $\rho_i:G_\Q\lra GL_2(\C),\ i=1,2$ are odd irreducible Artin representations. 
Let $f_i$ be the elliptic cusp form attached to $\rho_i$ explained as above.  
Then $L(s,\pi_p)=L_p(s,\pi_{f_1})L_p(s,\pi_{f_2})$ for all $p\nmid N$ and hence $\pi_F$ is of endoscopic type. 
Conversely, if $\pi_F$ is of endoscopic type, then there exist elliptic cusp forms $f_1, f_2$ such that $L(s,\pi_p)=L_p(s,\pi_{f_1})L_p(s,\pi_{f_2})$. 
It follows from the coefficients of $p^{-4s}$ of the local L-factors that each $f_i$ is of weight one. 
Hence we have $\rho_F=\rho_{f_1}\oplus\rho_{f_2}$ where $\rho_{f_i}$ is the Artin representation attached to $f_i$. 

Finally we remark that the independence of the once fixed embedding $\mathcal{O}_{\lambda_\ell}\hookrightarrow \C$ to 
our $\rho_F$ in (\ref{artin-rep-eq}) follows from the proof of Proposition \ref{mod-p-galois} and 
Chebotarev density theorem. This proves the main theorem. 

\begin{corollary}(Ramanujan conjecture) \label{Ram}
Let $\pi_F=\pi_{\infty}\otimes\otimes_p' \pi_p$ be the cuspidal representation of 
$GSp_4$ attached to the real analytic Siegel cusp eigenform of weight $(2,1)$ with the eigenvalues 
$-\frac 5{12}$ and $0$ for the generators $\Delta_1$ and $\Delta_2$. Then under the assumptions in Theorem 1.1, $\pi_p$ is tempered for all $p$.
\end{corollary}
\begin{proof} By Theorem \ref{artin-rep}, there exists the Artin representation $\rho_F : G_\Q\lra GSp_4(\C)$ such that
$det(I_4-\rho_F(\text{Frob}_p)T)=H_p(T)$ for almost all $p$. This shows that $\pi_p$ is tempered for almost all $p$.
For any one or two dimensional irreducible representation $\sigma: G_\Q\lra GL_r(\C)$, $r=1,2$, with solvable image, let $\pi_\sigma$ be the cuspidal representation attached to $\sigma$. Let $\rho_F=\rho_\infty\otimes \otimes_p' \rho_p$. 
Proposition A.1 of \cite{martin} extends to the Rankin-Selberg $L$-function $L(s,\sigma\times \rho_F)$.
Hence $L(s,\sigma_p\times \rho_p)=L(s,\pi_{\sigma,p}\times\pi_p)$ for all $p$. 

If $\pi_p$ is non-tempered and unitary, we have the following classification of $\pi_p$ (cf. \cite{RS}, Appendix):
(1) $\pi_p$ is the Langlands quotient of
$Ind_{P(\Q_p)}^{G(\Q_p)}\, \eta |det|^{\alpha}\rtimes \mu|\, |^{-\alpha}$, $0<\alpha<\frac 12$, where $\mu$ is a unitary character, and 
$\eta$ is a unitary supercuspidal representation of $GL_2(\Q_p)$ with $\omega_\eta=1$; (2) $\pi_p$ is the Langlands' quotient of $Ind_{Q(\Q_p)}^{G(\Q_p)}\, \xi|\ |^\alpha\rtimes \eta|det|^{-\frac {\alpha}2}$, $0<\alpha<1$, where $\xi^2=1$, $\xi\ne 1$, and $\eta$ is a unitary supercuspidal representation of 
$GL_2(\Q_p)$ such that $\xi\otimes\eta\simeq \eta$, i.e., $\eta$ is of dihedral type; (3) $\pi_p$ is the Langlands' quotient of $Ind_{B(\Q_p)}^{G(\Q_p)}\, \chi(\mu|\ |^\alpha, \nu|\ |^\beta, \xi|\, |^{-\frac {\alpha+\beta}2})$, where $\mu, \nu, \xi$ are unitary characters, and $\alpha>0$ or $\beta>0$. 

If $\pi_p$ is the Langlands quotient of $Ind_{Q(\Q_p)}^{G(\Q_p)}\, \xi|\ |^\alpha\rtimes \eta|det|^{-\frac {\alpha}2}$, then its lift to 
$GL_4(\Q_p)$ is $\eta|det|^{-\frac {\alpha}2}\oplus \eta|det|^{\frac {\alpha}2}$.
Now, we can write $\eta=\eta'|det|^{it}$ where $t\in \Bbb R$, and the central character of $\eta'$ is of finite order.
We choose a Galois representation $\sigma: G_\Q\lra GL_2(\C)$ such that $\pi_{\sigma,p}=\tilde{\eta'}$. 
[It can be done as follows: $\eta'$ is determined by a finite order character $\gamma$ of a quadratic extension $k/\Bbb Q_p$. By Grunwald-Wang theorem, we can choose a Hecke character $\chi$ of a quadratic extension $F/\Bbb Q$ such that $F_p=k$, and $\chi_p=\gamma$.
Hence $\chi$ gives rise to the Galois representation $\sigma$.]
Then $L(s,\sigma_p\times\rho_p)=L(s,\tilde{\eta'}\times\pi_p)$. The left hand side is holomorphic for 
$Re(s)>0$. On the other hand, the right hand side is $L(s-\frac {\alpha}2+it,\eta'\times\tilde{\eta'})L(s+\frac {\alpha}2+it,\eta'\times\tilde{\eta'})$, which 
has a pole at $s=\frac {\alpha}2-it$. This is a contradiction. 

If $\pi_p$ is the Langlands quotient of
$Ind_{P(\Bbb Q)}^{G(\Bbb Q)}\, \eta |det|^\alpha\rtimes \mu|\, |^{-\alpha}$, its lift to $GL_4(\Q_p)$ is $\mu|\,|^{-\alpha}\oplus \mu|\,|^{\alpha}\oplus \mu\otimes\eta$. Let $\mu=\mu'|\,|^{it}$, where $\mu'$ is of finite order. By choosing a Dirichlet character $\sigma$ such that $\sigma_p={\mu'}^{-1}$, we can deduce a contraction. If $\pi_p$ is the Langlands quotient of $Ind_{B(\Q_p)}^{G(\Q_p)}\, \chi(\mu|\ |^\alpha, \nu|\ |^\beta, \xi|\, |^{-\frac {\alpha+\beta}2})$, then its lift to $GL_4(\Q_p)$ is 
$\xi|\,|^{-\frac {\alpha+\beta}2}\oplus \mu\xi|\,|^{\frac {\alpha-\beta}2}\oplus \nu\xi|\,|^{-\frac {\alpha-\beta}2}\oplus \mu\nu\xi|\,|^{\frac {\alpha+\beta}2}$. Since one of $\pm\frac {\alpha+\beta}2$ or $\pm\frac {\alpha-\beta}2$ is positive, again we deduce a contradiction.
Hence $\pi_p$ is tempered for all $p$.
\end{proof}
The following proposition is due to R. Schmidt \cite{schmidt1}, Corollary 3.2.3.
\begin{prop} Let $F$ be a holomorphic Siegel cusp form of weight $(k_1,k_2)$, and $\pi_F=\pi_\infty\otimes\otimes_p' \pi_p$ be the associated cuspidal representation of $GSp_4/\Q$. Then $\pi_\infty$ is a subquotient of $\text{Ind}_B^G\, \chi(\mu_1,\mu_2,\sigma)$, where $\mu_1, \mu_2, \sigma$ are characters of $\Bbb R^\times$ such that $\sigma(x)=x^{\frac {3-k_1-k_2}2} \, \text{for $x>0$}$, and
$$\mu_1(x)=\begin{cases} |x|^{k_2-2}, &\text{if $k_2$ even}\\ |x|^{k_2-2}\sgn(x), &\text{if $k_2$ odd}\end{cases},\quad
\mu_2(x)=\begin{cases} |x|^{k_1-1}, &\text{if $k_2$ even}\\ |x|^{k_1-1}\sgn(x), &\text{if $k_2$ odd}\end{cases}.
$$
\end{prop}

Using it, we can prove
\begin{prop} \label{no-hol}
Let $(k_1,k_2)$ be a pair of integers $k_1\ge k_2\ge 0$. Then 
there are no holomorphic Siegel cusp forms of weight $(k_1,k_2)$ which give rise to the Artin representations.
\end{prop}
\begin{proof} Let $F$ be a holomorphic Siegel cusp form of weight $(k_1,k_2)$ which is a Hecke eigenform.
We may assume that $k_2>0$ by the holomorphy. Let 
$\pi_F=\pi_\infty\otimes \otimes_p' \pi_p$ be the associated cuspidal representation of $GSp_4/\Bbb Q$. Then by the above proposition, 
$\pi_\infty$ is a subquotient of ${\rm Ind}_B^G \, \chi(\mu_1,\mu_2,\sigma)$, where $\mu_1,\mu_2,\sigma$ are as in the above proposition.
If $\pi_F$ corresponds to an Artin representation $\rho: G_{\Bbb Q}\longrightarrow GSp_4(\Bbb C)$, then 
$L(s,\pi_p)=L(s,\rho_p)$ for almost all $p$. By \cite{martin}, Proposition A.1, $L(s,\pi_\infty)=L(s,\rho_\infty)$. So 
the Langlands parameter of $\pi_\infty$ is $\phi: W_{\Bbb R}\lra GSp_4(\C)$,
$\phi(z)=I_4$, and $\phi(j)$ is conjugate to $\diag(\pm 1,\pm 1,\mp 1,\mp 1)$. Therefore, $\pi_\infty={\rm Ind}_B^G\, \chi(\epsilon_1,\epsilon_2,\epsilon_0)$, where $\epsilon_i=1$ or $\sgn$ from the discussion in Section \ref{inf}. This is a contradiction. 
\end{proof}

\section{Symmetric cube of elliptic cusp forms of weight 1}\label{sym-cube}

Let $\pi$ be a cuspidal representation of $GL_2/\Q$ which corresponds to the weight 1 new form with respect to $\Gamma_0(N)$ with the central character $\epsilon$.
Let $\rho$ be the Galois representation $\rho: G_{\Q}\longrightarrow GL_2(\Bbb C)$ which corresponds to $\pi$ by Deligne-Serre theorem \cite{d&s}. Then ${\rm Sym}^3(\pi)$ be an automorphic representation of $GL_4/\Q$ with the central character $\epsilon^3$ \cite{Kim-Sh}. 

If $\pi$ is of dihedral or tetrahedral type, ${\rm Sym}^3(\pi)=\pi_1\boxplus \pi_2$ for cuspidal representations $\pi_1, \pi_2$ of 
$GL_2/\Q$.
Otherwise, i.e., if $\pi$ is of octahedral or icosahedral type, ${\rm Sym}^3(\pi)$ is cuspidal. 
In all cases, since $L(s, {\rm Sym}^3(\pi), \wedge^2\otimes\epsilon^{-3})$ has a pole at $s=1$, by the result of Jacquet, Piatetski-Shapiro and Shalika (cf. \cite{AS}, \cite{AS1}), 
there exists a generic cuspidal representation $\tau$ of $GSp_4(\A)$ with the central character $\epsilon^3$ whose transfer to 
$GL_4(\A)$ is 
${\rm Sym}^3(\pi)$. The Langlands parameter of $\pi_{\infty}$ is
$$\phi: W_{\Bbb R}\longrightarrow GL_2(\Bbb C),\quad \phi(z)=I_2, \quad \phi(j)=\begin{pmatrix} 0&1\\1&0\end{pmatrix}.
$$
Hence the Langlands parameter of ${\rm Sym}^3(\pi_\infty)$ is 
$${\rm Sym}^3(\phi): W_{\Bbb R}\longrightarrow GL_4(\Bbb C),\quad {\rm Sym}^3(\phi)(z)=I_4, \quad {\rm Sym}^3(\phi)(j)=J',
$$
where
$J'=\begin{pmatrix} 0&0&0&1\\0&0&1&0\\0&1&0&0\\1&0&0&0\end{pmatrix}$. Let $P=\frac 12\begin{pmatrix} 1&-1&0&0\\0&0&1&1\\0&0&-1&1\\1&1&0&0\end{pmatrix}$. Then $P^{-1}={}^t P$, $P^{-1}J' P=\diag(1,-1,-1,1)$, and ${}^t P\begin{pmatrix} 0&I_2\\-I_2&0\end{pmatrix} P=\frac 12\begin{pmatrix} 0&I_2\\-I_2&0\end{pmatrix}$. Therefore, $J'$ is conjugate to $\diag(1,-1,-1,1)$ in $GSp_4(\C)$. Note that 
$L(s,{\rm Sym}^3(\pi_\infty))=\Gamma_{\Bbb C}(s)^2$, where $\Gamma_{\Bbb C}(s)=2(2\pi)^{-s}\Gamma(s)$.
Since the Langlands parameter of $\tau_{\infty}$ is ${\rm Sym}^3(\phi)$, and from the discussion in Section 5, 
$\tau_\infty={\rm Ind}_B^G \, \chi(1,\sgn,\sgn)$.
Now choose the highest weight vector in the $K_\infty$-type (2,1) in $\tau_{\infty}$, and 
take an automorphic form $\phi$ whose archimedean component is the highest weight vector. Then $\phi$ corresponds to a real analytic Siegel modular form $F$ on the upper half-space taking values in some two-dimensional vector space so that $\pi_F$ is in the same $L$-packet as in $\tau$:
More precisely, let $(f_0(k), f_1(k))$ be the first row of the $2\times 2$ matrix ${\rm det}(k) u(k)$ for $k\in K_\infty$, where $u(k)$ is given by the isomorphism $u: K_\infty\simeq U(2)$  as in Section \ref{inf}.
Then $V=\C f_0\oplus \C f_1$ forms an irreducible $K_\infty$-representation with the highest weight $(2,1)$, and $f_0$ is the highest weight vector, and $f_1$ is the lowest weight vector. Now using the Iwasawa decomposition $GSp_4(\Bbb R)=B(\Bbb R)K_\infty$, define the scalar-valued function
$\tilde\phi: GSp_4(\Bbb R)\longrightarrow \Bbb C$ by 
$$\tilde\phi(tuk)=\sgn(t_2)\sgn(t_0)|t_0|^{-\frac 32}|t_1|^2|t_2| f_0(k), \quad t=\diag(t_1,t_2,t_0 t_1^{-1}, t_0 t_2^{-1}).
$$
We also can define the vector-valued function $\phi: GSp_4(\Bbb R)\longrightarrow V$ by 
$$\phi(tuk)=\sgn(t_2)\sgn(t_0)|t_0|^{-\frac 32}|t_1|^2|t_2| (f_0(k)e_1+f_1(k)e_2).
$$
(Here $V$ can be identified with ${\rm Sym}({\rm St}_2)\otimes {\rm det}({\rm St}_2)$ in Section 3.1.)
Let $\lambda$ be the algebraic representation of $GL_2$ on $V$ as in Section 3.1. For $Z\in \mathcal H_2$, let $Z=g I$, and $F(Z)=\lambda(J(g,I))\phi(g)$. Then as in Section 3.3, we can show easily that $F$ is a real analytic Siegel cusp form of weight $(2,1)$.

This provides the existence of infinitely many real analytic Siegel cusp forms of weight $(2,1)$ with integral Hecke polynomials. Note that this is an unconditional result. We summarize our result as follows:
\begin{theorem} Let $f$ be a cusp form of weight one with respect to $\Gamma_0(N)$ with the central character $\epsilon$. Suppose $f$ is a Hecke eigenform. Then there exists a real analytic Siegel cusp form $F$ of weight $(2,1)$ with the eigenvalues $-\frac 5{12}$ and $0$ for the generators $\Delta_1$ and $\Delta_2$, and with integral Hecke polynomials such that 
${\rm Sym}^3(\pi_f)$ is the transfer of $\pi_F$.
\end{theorem}

Note that the image of ${\rm Sym}^3(\rho): G_{\Q}\longrightarrow GL_4(\Bbb C)$, is in $GSp_4(\Bbb C)$ (cf. \cite{GW}, page 244), and 
the parameter ${\rm Sym}^3(\rho): G_{\Q}\longrightarrow GSp_4(\Bbb C)$ corresponds to $\pi_F$.

\section{Siegel cusp forms of solvable type} \label{solvable}

Let $\rho: G_\Q\lra GSp_4(\C)$ and $\bar\rho: G_\Q\lra PGSp_4(\C)$ be as in the introduction.
In this section, we recall K. Martin's result on the strong Artin conjecture for $\rho$ \cite{martin}. He showed the strong Artin conjecture when ${\rm Im}(\bar\rho)$ is a solvable group, $E_{16}\rtimes C_5$, where 
$E_{16}\simeq (\Z/2\Z)^{\oplus 4}$ is the
elementary abelian group of order 16 and $C_5$ is the cyclic group of order 5. 
We denote by $Q_8$ (resp. $D_8$), the quaternion group of order $8$ (resp. dihedral group of order $8$). 
Note that $D_8$ here is denoted as $D_4$ in \cite{JLY}, page 35.  

We will give an explicit example of such $\rho$ which is taken from Section 5 of \cite{martin}, but 
we make a slight change for the reader's convenience. 

Let $\zeta_{11}$ be a primitive 11-th root of unity, and $\alpha_i=\zeta_{11}^i+\zeta_{11}^{-i}$. 
Let $E=\Q(\alpha_1)$ be a quintic extension over $\Q$. Let 
$$K=E(\sqrt{a}),\quad M=E(\sqrt[4]{\alpha_1},\ \sqrt{-1}),
$$
where $a=(1+\frac{1}{\sqrt{u}})(1+\frac{1}{\sqrt{v}})$, $u=1+\alpha^2_3,\ v=1+\alpha_1^2+\alpha_1^2 \alpha_3^2$. Then $K/E, M/E$ are Galois extensions with
${\rm Gal}(K/E)\simeq Q_8,\ {\rm Gal}(M/E)\simeq D_8$. (For $K/E$, let 
$(\alpha,\beta,\gamma)=(\alpha_3,0,\alpha_1)$ in Remark of \cite{JLY}, page 135, and for $M/E$, see  
Theorem 2.2.7 of \cite{JLY}, page 35.)

Let $L=E(\sqrt{u}, \sqrt{v}, \sqrt[4]{\alpha_1},\ \sqrt{-1}))$ and $L_0=
E(\sqrt{u},\ \sqrt{v},\ \sqrt{\alpha_1},\ \sqrt{-1}))$.
Then $L$ is a subextension of $KM$ of index 2 which corresponds to a subgroup of the central product $Q_8D_8={\rm Gal}(KM/E)$ of $Q_8$ and $D_8$. Then
$${\rm Gal}(L/\Q)\simeq ((\Z/2\Z)^{\oplus 2}\times D_8)\rtimes C_5,\quad
{\rm Gal}(L_0/\Q)\simeq ((\Z/2\Z)^{\oplus 2}\times D_8/\{\pm 1\})\rtimes C_5\simeq E_{16}\rtimes C_5.
$$ 
Note that $D_8/\{\pm 1\}$ splits, and hence it is isomorphic to $(\Z/2\Z)^{\oplus 2}$. 

Therefore one has ${\rm Gal}(L/\Q)\hookrightarrow GSp_4(\C)$ by Section 5 of \cite{martin}, and it gives rise to an Artin representation $\rho: G_\Q\lra GSp_4(\C)$. Further $\bar\rho: G_\Q\lra PGSp_4(\C)$ gives ${\rm Gal}(L_0/\Q)\hookrightarrow PGSp_4(\C)$.
An explicit description of $\rho$ is given as follows. 
Let $J_1=
\begin{pmatrix}
1 & 0 \\
0 & -1
\end{pmatrix}$  and $J_2=
\begin{pmatrix}
0 & 1 \\
1 & 0
\end{pmatrix}$. 
Let
$$A_1=\begin{pmatrix}
J_1 & 0_2 \\
0_2 & J_1
\end{pmatrix},\ 
A_2=
\begin{pmatrix}
J_2 & 0_2 \\
0_2 & -J_2
\end{pmatrix},\ 
A_3=
\begin{pmatrix}
\sqrt{-1}J_2 & 0_2 \\
0_2 & \sqrt{-1}J_2
\end{pmatrix},$$
$$A_4=
\begin{pmatrix}
0_2 & \sqrt{-1}J_2 \\
\sqrt{-1}J_2 & 0_2
\end{pmatrix},\ 
A_5=\diag(1,-1,-1,1),
$$ 
$$T=-\frac{1+\sqrt{-1}}{2}
\begin{pmatrix}
-\sqrt{-1} & 0 & 0 & \sqrt{-1} \\
0 & 1 & 1 & 0 \\
1 & 0 &0 & 1 \\
0 & -\sqrt{-1} & \sqrt{-1} & 0 
\end{pmatrix}.$$
Then $\langle A_1,A_2,A_3,A_4,A_5 \rangle\simeq {\rm Gal}(L/E)$ and $\langle T \rangle\simeq C_5$ acts on ${\rm Gal}(L/E)$ by conjugation. 
The Galois action $\sqrt{-1}\mapsto -\sqrt{-1}$ on $L$ (and also on $L_0$) corresponds to $A_5$. Clearly, 
\newline $\langle A_1,A_2,A_3,A_4,A_5 \rangle/\{\pm I_4\}\simeq E_{16}$.

Since the complex conjugate acts on $L_0$ non-trivially, $\rho(c)\neq \pm I_4$. Hence $\rho: G_\Q\lra GSp_4(\C)$ is symplectically odd. K. Martin showed that $\rho$ is modular. So it corresponds to a cuspidal automorphic representation 
$\widetilde{\Pi}$ of $GL_4(\A)$, and descends to a cuspidal representation $\Pi$ of $GSp_4(\A)$. Since $L(s,\Pi_p)=L(s,\rho_p)$ for almost all $p$, by \cite{martin}, Appendix, $L(s,\rho_\infty)=L(s,\Pi_\infty)$. Since $L(s,\rho_\infty)=\Gamma_{\Bbb C}(s)^2$,
the Langlands parameter of $\Pi_\infty$ is $\phi: W_{\Bbb R}\lra GSp_4(\C)$, which is the composition of $i: W_{\Bbb R}\lra G_\R\hookrightarrow G_\Q$ and $\rho$. Hence 
$\phi(z)=I_4$ and $\phi(j)\stackrel{GSp_4(\C)}{\sim} diag(1,-1,-1,1)$. So $\Pi_\infty={\rm Ind}_B^G\, \chi(1,\sgn,\sgn)$. 
As in Section \ref{sym-cube}, there exists a real analytic Siegel cusp form of weight $(2,1)$ which corresponds 
to the Galois representation $\rho$. This gives the existence of Siegel cusp form of weight $(2,1)$ with the eigenvalues $-\frac 5{12}$ and $0$ for the generators $\Delta_1$ and $\Delta_2$ and with integral Hecke polynomials, which does not come from $GL_2$ form.

\end{document}